\documentclass[12pt]{amsart}
\usepackage[all]{xy}

\newcommand{\Cdb}{\mbox{$\mathbb{C}$}}

\newcommand{\Rdb}{\mbox{$\mathbb{R}$}}

\newcommand{\Tdb}{\mbox{$\mathbb{T}$}}

\newcommand{\Zdb}{\mbox{$\mathbb{Z}$}}

\newcommand{\B}{\mbox{${\mathcal B}$}}

\newcommand{\E}{\mbox{${\mathcal E}$}}

\newcommand{\U}{\mbox{${\mathcal U}$}}

\newcommand{\norm}[1]{\Vert#1\Vert}
\newcommand{\bignorm}[1]{\bigl\Vert#1\bigr\Vert}
\newcommand{\Bignorm}[1]{\Bigl\Vert#1\Bigr\Vert}

\newcommand{\cbnorm}[1]{\Vert#1\Vert_{cb}}

\newcommand{\matnorm}[1]{\Vert#1\Vert_{{\rm Mat-}\gamma}}

\newcommand{\matnormR}[1]{\Vert#1\Vert_{{\rm Mat-}R}}

\newcommand{\Gnorm}[1]{\Vert#1\Vert_{G}}
\newcommand{\bigGnorm}[1]{\bigl\Vert#1\bigr\Vert_{G}}
\newcommand{\BigGnorm}[1]{\Bigl\Vert#1\Bigr\Vert_{G}}
\newcommand{\lnorm}[1]{\Vert#1\Vert_{\ell}}
\newcommand{\biglnorm}[1]{\bigl\Vert#1\bigr\Vert_{\ell}}
\newcommand{\Biglnorm}[1]{\Bigl\Vert#1\Bigr\Vert_{\ell}}

\newcommand{\ra}[1]{{\rm Rad}(#1)}
\newcommand{\rara}[1]{{\rm Rad}({\rm Rad}(#1))}

\newcommand{\rarara}[1]{{\rm Rad}({\rm Rad}({\rm Rad}(#1)))}

\newcommand{\triple}[1]{\vert\vert\vert{#1}\vert\vert\vert}

\newcommand{\bH}{\overline{H}}

\newcommand{\h}{\mathop{\otimes}\limits^{2}}

\newtheorem{theorem}{Theorem}[section]
\newtheorem{lemma}[theorem]{Lemma}
\newtheorem{corollary}[theorem]{Corollary}
\newtheorem{proposition}[theorem]{Proposition}
\newtheorem{definition}[theorem]{Definition}
\theoremstyle{remark}
\newtheorem{remark}[theorem]{\bf Remark}
\theoremstyle{definition}

\numberwithin{equation}{section}

\begin{document}
\baselineskip 15pt

\title[]{$\gamma$-bounded representations of amenable groups}

\author{Christian Le Merdy}
\address{D\'epartement de Math\'ematiques\\ Universit\'e de  Franche-Comt\'e
\\ 25030 Besan\c con Cedex\\ France}
\email{clemerdy@univ-fcomte.fr}


\thanks{The author is supported by the research
program ANR-06-BLAN-0015.}

\begin{abstract} Let $G$ be an amenable group, let $X$ be a Banach
space and let $\pi\colon G\to B(X)$ be a bounded representation.
We show that if the set $\{\pi(t)\, :\, t\in G\}$ is
$\gamma$-bounded then $\pi$ extends to a bounded homomorphism
$w\colon C^{*}(G)\to B(X)$ on the group $C^{*}$-algebra of $G$.
Moreover $w$ is necessarily $\gamma$-bounded. This extends to the
Banach space setting a theorem of Day and Dixmier saying that any bounded
representation of an amenable group on Hilbert space is
unitarizable. We obtain additional results and complements when 
$G=\Zdb,\,\Rdb$ or $\Tdb$, and/or when $X$
has property $(\alpha)$.

\end{abstract}

\maketitle


\bigskip

\section{Introduction.}
The notions of $R$-boundedness and $\gamma$-boundedness play a
prominent role in various recent developments of operator valued
harmonic analysis and multiplier theory, see for example \cite{W1, SW, AB, BLM, H,
HW1, HW2, KW}. These notions are also now central in the closely
related fields of functional calculi (see \cite{KaW1, DHP, KuW}), abstract
control theory in Banach spaces \cite{HK1, HK2}, or vector valued stochastic integration, see
\cite{VNVW2} and the references therein. This paper is devoted to another aspect of harmonic
analysis, namely Banach space valued group representations. Our results 
will show that $\gamma$-boundedness
is the key concept to understand certain behaviors of such
representations.

Throughout we let $G$ be a locally compact group, we let $X$ be a complex
Banach space and we let $B(X)$ denote the Banach algebra of all
bounded operators on $X$. By a representation of $G$ on $X$, we
mean a strongly continuous mapping $\pi\colon G\to B(X)$ such that
$\pi(tt')=\pi(t)\pi(t')$ for any $t,t'$ in $G$, and $\pi(e)=I_X$.
Here $e$ and $I_X$ denote the unit of $G$ and the identity
operator on $X$, respectively. We say that $\pi$ is bounded if
moreover $\sup_{t\in G}\norm{\pi(t)}<\infty$. Assume that $G$ is
amenable and that $X=H$ is a Hilbert space. Then it follows from
the Day-Dixmier unitarization Theorem (see e.g. \cite[Chap. 0]{Pis2})
that any bounded representation of
$G$ on $H$ extends to a bounded homomorphism $C^{*}(G)\to B(H)$
from the group $C^*$-algebra $C^{*}(G)$ into $B(H)$. In general
this extension property is no longer possible when $H$ is replaced by an
arbitrary Banach space. To see a simple example, let $G$ be an infinite abelian group,
let $1\leq p<\infty$ and let
$\lambda_p\colon G\to B(L^{p}(G))$ be the regular
representation defined by letting $\bigl[\lambda_p(t)f](s)
=f(s-t)$ for any $f\in L^{p}(G)$. Recall that $C^{*}(G)=C_0(\widehat{G})$, 
where $\widehat{G}$ denotes the dual group of $G$. Hence if
$\lambda_p$ extends to a bounded homomorphism $C^{*}(G)\to B(L^{p}(G))$, 
then any function in $C_0(\widehat{G})$ is a bounded Fourier multiplier on $L^p(G)$.  
As is well-known, this implies that $p=2$, see e.g. \cite[Thm. 4.5.2]{La}. (See
also Corollary \ref{5Regular} for more on this.) 
This leads to the problem of finding conditions on a Banach
space representation $\pi\colon G\to B(X)$ ensuring that its
extension to a bounded homomorphism $C^{*}(G)\to B(X)$ is indeed
possible.

We recall the definitions of $\gamma$-boundedness and
$R$-boundedness. The latter is more classical (see \cite{CPSW}), but the two notions 
are completely similar. Let
$(g_k)_{k\geq 1}$ be a sequence of complex valued, independent
standard Gaussian variables on some probability space $\Sigma$.
For any $x_1,\ldots,x_n$ in $X$, we let
$$
\Bignorm{\sum_k g_k\otimes
x_k}_{G(X)}\,=\,\Bigl(\int_{\Sigma}\Bignorm{\sum_k g_k(\lambda) \,
x_k}^{2}_{X}\, d\lambda\,\Bigr)^{\frac{1}{2}}.
$$
Next we say that a set $F\subset B(X)$ is $\gamma$-bounded if
there is a constant $C\geq 0$ such that for any finite families
$T_1,\ldots, T_n$ in $F$, and $x_1,\ldots,x_n$ in $X$, we have
$$
\Bignorm{\sum_k g_k\otimes T_k x_k}_{G(X)}\,\leq\, C\,
\Bignorm{\sum_k g_k\otimes x_k}_{G(X)}.
$$
In this case, we let $\gamma(F)$ denote the smallest possible $C$.
This constant is called the $\gamma$-bound of $F$. Now let
$(\varepsilon_k)_{k\geq 1}$ be a sequence of independent
Rademacher variables on some probability space. Then replacing the
sequence $(g_k)_{k\geq 1}$ by the sequence $(\varepsilon_k)_{k\geq
1}$ in the above definitions, we obtain the notion of
$R$-boundedness. The corresponding $R$-bound constant of $F$ is denoted by $R(F)$. 
Using the symmetry of Gaussian variables, it is easy to see
(and well-known) that any $R$-bounded set $F\subset B(X)$ is automatically $\gamma$-bounded,
with $\gamma(F)\leq R(F)$. If further $X$ has a
finite cotype, then Rademacher averages and Gaussian averages are
equivalent (see e.g. \cite[Chap. 3]{Pis3}), hence the notions of $R$-boundedness and
$\gamma$-boundedness are equivalent. Clearly any $\gamma$-bounded set
is bounded and if $X$ is isomorphic to a Hilbert space, then any bounded 
set is $\gamma$-bounded. We recall that conversely if $X$ is not isomorphic to a Hibert
space, then there exist bounded sets $F\subset B(X)$ which are not
$\gamma$-bounded (see \cite[Prop. 1.13]{AB}).

Our main result asserts that if $G$ is amenable and if $\pi\colon
X\to B(X)$ is a representation such that $\{\pi(t)\, :\, t\in G\}$
is $\gamma$-bounded, then there exists a (necessarily unique)
bounded homomorphism $w\colon C^{*}(G)\to B(X)$ extending $\pi$
(see Definition \ref{2Def} for the precise meaning). Moreover $w$
is $\gamma$-bounded, i.e. it maps the unit ball of $C^{*}(G)$ into
a $\gamma$-bounded set of $B(X)$.

If $X$ has property $(\alpha)$, we obtain the
following analog of the Day-Dixmier unitarization Theorem: a
representation $\pi\colon G\to B(X)$ extends to a bounded
homomorphism $C^{*}(G)\to B(X)$ if and only if $\{\pi(t)\, :\,
t\in G\}$ is $\gamma$-bounded. As an illustration, consider the
case $G=\Zdb$ and recall that $C^{*}(\Zdb)=C(\Tdb)$. Let $T\colon
X\to X$ be an invertible operator on a Banach space with property
$(\alpha)$. We obtain that there exists a constant $C\geq 1$ such that
$$
\Bignorm{\sum_k c_k T^{k}}\,\leq\, C\,\sup\Bigl\{\Bigl\vert\sum_k
c_k z^{k}\Bigr\vert\, :\, z\in\Cdb,\, \vert z\vert =1\Bigr\}
$$
for any finite sequence $(c_k)_{k\in\footnotesize{\Zdb}}$ of
complex numbers, if and only if the set
$$
\bigl\{T^k\, :\, k\in\Zdb\bigr\}\qquad\hbox{is}\ \gamma\hbox{-bounded}.
$$

The main result presented above is established in Section 4. Its proof makes crucial use of the transference 
methods available on amenable groups (see \cite{CW}) and 
of the Kalton-Weis $\ell$-spaces introduced in the unpublished paper \cite{KaW2}. 
Sections 2 and 3 are devoted to preliminary results and background on these spaces 
and on group representations. In Section 5 we give a proof of the following result:
if a Banach space $X$ has property $(\alpha)$, then any bounded 
homomorphism $w\colon A\to B(X)$ defined on a nuclear $C^*$-algebra $A$ is automatically 
$R$-bounded (and even matricially $R$-bounded). This result is due 
to \'Eric Ricard (unpublished). In the case when $A$ is abelian,
it goes back to De Pagter-Ricker \cite{DR1} (see also \cite{KL}). 
Section 6 contains examples and illustrations, some of them using the above theorem.
We pay a special attention to the $\gamma$-bounded representations of the classical abelian 
groups $\Zdb,\Rdb,\Tdb$. 
 
\bigskip
We end this introduction with some notation and general references. 
First, we will use vector valued integration and Bochner $L^p$-spaces for which 
we refer to \cite{DU}. We let $G(X)\subset L^{2}(\Sigma;X)$
be the closed subspace spanned by the finite sums $\sum_k g_k\otimes x_k$, with $x_k\in X$. 
Next the space 
$\ra{X}$ is defined similarly, using the Rademacher sequence $(\varepsilon_k)_{k\geq 1}$.
For any $n\geq 1$, we let ${\rm Rad}_n(X)\subset \ra{X}$ be the subspace of all sums
$\sum_{k=1}^{n} \varepsilon_k\otimes x_k$. It follows from classical duality 
on Bochner spaces that
we have a natural isometric isomorphism
\begin{equation}\label{NSecond}
{\rm Rad}_n(X)^{**}\, =\, {\rm Rad}_n(X^{**}).
\end{equation}
Second, we refer to \cite{F} for general background on classical harmonic analysis. 
Given a locally compact group $G$, we let $dt$ denote a fixed left Haar measure on $G$.
For any $p\geq 1$, we let $L^p(G)= L^p(G,dt)$ denote the
corresponding $L^p$-space. We recall that the convolution on $G$ makes $L^{1}(G)$ a Banach
algebra. Finally we will use basic facts on $C^*$-algebras and Hilbert space representations, for which 
\cite{Pis2} and \cite{Pa} are relevant references. 

For any Banach spaces $X,Y$, we let $B(Y,X)$ denote the space of all bounded operators from $Y$ into $X$, 
equipped with the operator norm, and we set $B(X)=B(X,X)$. 
Given any set $V$, we let $\chi_V$ denote the indicator function of $V$.

\medskip
\section{Preliminaries on $\gamma$-bounded representations.}
We let $M_{n,m}$ denote the space of $n\times m$ scalar matrices equipped 
with its usual operator 
norm. We start with the following well-known tensor extension property,
for which we refer e.g. to  \cite[Cor. 12.17]{DJT}.

\begin{lemma}\label{2Inv}
Let $a=[a_{ij}]\in M_{n,m}$ and let
$x_1,\ldots, x_m\in X$. Then
$$
\Bignorm{\sum_{i,j} a_{ij} g_i\otimes
x_j}_{G(X)}\,\leq\,\norm{a}_{M_{n,m}}\, \bignorm{\sum_{j}
g_j\otimes x_j}_{G(X)}.
$$
\end{lemma}

This result does not remain true if we replace Gaussian variables
by Rademacher variables and this defect is the main reason why it
is sometimes easier to deal with $\gamma$-boundedness than with
$R$-boundedness.

An extremely useful property proved in \cite[Lem. 3.2]{CPSW} is
that if $F\subset B(X)$ is any $R$-bounded set, then its strongly closed
absolute convex hull $\overline{\rm aco}(F)$ is $R$-bounded as
well, with an estimate $R(\overline{\rm aco}(F))\leq 2 R(F)$. It
turns out that a similar property holds for $\gamma$-bounded sets
without the extra factor 2.

\begin{lemma}\label{2Conv}
Let $F\subset B(X)$ be any $\gamma$-bounded set. Then its closed
absolute convex hull $\overline{\rm aco}(F)$ with respect to the
strong operator topology is $\gamma$-bounded as well, and
$$
\gamma(\overline{\rm aco}(F))=\gamma(F).
$$
\end{lemma}

\begin{proof} Consider the set
$$
\widetilde{F}=\bigl\{zT\, :\, T\in F,\ z\in\Cdb,\ \vert z\vert\leq
1\bigr\}.
$$
Applying Lemma \ref{2Inv} to diagonal matrices, we see that
$\widetilde{F}$ is $\gamma$-bounded and that
$\gamma(\widetilde{F})=\gamma(F)$. Moreover ${\rm aco}(F)$ is equal to 
${\rm co}(\widetilde{F})$, the
convex hull of $\widetilde{F}$. Hence the argument in \cite[Lem.
3.2]{CPSW} shows that ${\rm aco}(F)$ is $\gamma$-bounded and that
$\gamma({\rm aco}(F)) = \gamma(\widetilde{F})$. The result follows
at once.
\end{proof}

Let $Z$ be an arbitrary Banach space. Following \cite{KL}, we say
that a bounded linear map $v\colon Z\to B(X)$ is $\gamma$-bounded
(resp. $R$-bounded) if the set
$$
\bigl\{v(z)\,:\, z\in Z,\ \norm{z}\leq 1\bigr\}
$$
is $\gamma$-bounded (resp. $R$-bounded). In this case, we let
$\gamma(v)$ (resp. $R(v)$) denote the $\gamma$-bound (resp. the
$R$-bound) of the latter set.

Next we say that a representation $\pi\colon G\to B(X)$ is
$\gamma$-bounded (resp. $R$-bounded) if the set
$$
\bigl\{\pi(t)\,:\, t\in G\bigr\}
$$
is $\gamma$-bounded (resp. $R$-bounded). In this case, we let
$\gamma(\pi)$ (resp. $R(\pi)$) denote the $\gamma$-bound (resp.
the $R$-bound) of the latter set.

For any bounded representation $\pi\colon G\to B(X)$, we let
$\sigma_\pi\colon L^{1}(G)\to B(X)$ denote the associated bounded
homomorphism defined by
$$
\sigma_\pi(k)=\,\int_{G} k(t)\pi(t)\, dt,\qquad k\in L^{1}(G),
$$
where the latter integral in defined in the strong sense. It turns
out that $\sigma_\pi$ is nondegenerate, that is,
\begin{equation}\label{2Ess}
{\rm Span}\bigl\{\sigma_\pi(k)x\, :\, k\in L^{1}(G), \ x\in
X\bigr\}
\end{equation}
is dense in $X$. Moreover, for every nondegenerate bounded
homomorphism $\sigma\colon L^{1}(G)\to B(X)$, there exists a
unique representation $\pi\colon G\to B(X)$ such that
$\sigma=\sigma_\pi$, see \cite[Lem. 2.4 and Rem. 2.5]{DR2}.

\begin{lemma}\label{2Equiv} Let $\pi\colon G\to B(X)$ be a bounded representation.
Then $\pi$ is $\gamma$-bounded if and only if $\sigma_\pi$ is
$\gamma$-bounded. Moreover $\gamma(\pi)=\gamma(\sigma_\pi)$ in
this case.
\end{lemma}

\begin{proof}
For any $k\in L^{1}(G)$ such that $\norm{k}_1\leq 1$, the operator
$\sigma_\pi(k)$ belongs to the strongly closed absolute convex hull of
$\{\pi(t)\, :\,t\in G\}$. Hence the `only if' part follows from
Lemma \ref{2Conv}, and we have $\gamma(\sigma_\pi)\leq
\gamma(\pi)$.

For the converse implication, we let $(h_\iota)_\iota$ be a
contractive approximate identity of $L^{1}(G)$. For any $t\in G$,
let $\delta_t$ denote the point mass at $t$. Then for any $k\in
L^{1}(G)$, and any $x\in X$, we have
\begin{align*}
\pi(t)\sigma_\pi (k) x \, & =\, \sigma_\pi(\delta_t * k)x\\
& =\,\lim_\iota \sigma_\pi(h_\iota *\delta_t * k)x\\
& =\,\lim_\iota \sigma_\pi(h_\iota *\delta_t)\sigma_\pi(k)x.
\end{align*}
Hence if we let $Y\subset X$ be the dense subspace defined by
(\ref{2Ess}), we have that
$$
\lim_\iota \sigma_\pi(h_\iota *\delta_t) y\,=\,\pi(t)y,\qquad y\in
Y, \ t\in G.
$$
Now assume that $\sigma_\pi$ is $\gamma$-bounded and let
$y_1,\ldots, y_n\in Y$ and $t_1,\ldots,  t_n\in G$. For any $\iota$
and any $k=1,\ldots,n$, we have $\norm{h_\iota *
\delta_{t_k}}_1\leq 1$. Hence
$$
\Bignorm{\sum_k g_k\otimes \sigma_\pi(h_\iota * \delta_{t_k})
y_k}_{G(X)}\,\leq\,\gamma(\sigma_\pi) \,\Bignorm{\sum_k g_k\otimes
y_k}_{G(X)}.
$$
Passing to the limit when $\iota\to\infty$, this yields
$$
\Bignorm{\sum_k g_k\otimes \pi(t_k)
y_k}_{G(X)}\,\leq\,\gamma(\sigma_\pi) \,\Bignorm{\sum_k g_k\otimes
y_k}_{G(X)}.
$$
Since $Y$ is dense in $X$, this implies that $\pi$ is
$\gamma$-bounded, with $\gamma(\pi)\leq\gamma(\sigma_\pi)$.
\end{proof}

Let $\lambda\colon G\to B(L^{2}(G))$ denote the left regular
representation. We recall that for any $f\in L^{2}(G)$,
$$
\lambda(t)f = \delta_t *f\qquad\hbox{and}\qquad
\sigma_\lambda(k)=k*f
$$
for any $t\in G$ and any $k\in L^{1}(G)$. The reduced
$C^*$-algebra of $G$ is defined as
$$
C^{*}_{\lambda}(G)=\overline{\sigma_\lambda\bigr(L^{1}(G)\bigr)}\,\subset
B(L^{2}(G)).
$$
We recall that $C^{*}_{\lambda}(G)$ is equal to the group
$C^{*}$-algebra $C^{*}(G)$ if and only if $G$ is amenable, see e.g. \cite[(4.21)]{Pat}. The
notion on which we will focus on in Section 4 and beyond is the following.

\begin{definition}\label{2Def} We say that a bounded representation
$\pi\colon G\to B(X)$ extends to a bounded homomorphism $w\colon
C^{*}_{\lambda}(G)\to B(X)$ if $w\circ\sigma_\lambda=\sigma_\pi$.
\end{definition}

Note that there exists a bounded operator $w\colon
C^{*}_{\lambda}(G)\to B(X)$ such that
$w\circ\sigma_\lambda=\sigma_\pi$ if and only if there is a
constant $C\geq 0$ such that
$$
\norm{\sigma_\pi(f)}\leq C\norm{\sigma_\lambda(f)},\qquad f\in
L^{1}(G),
$$
that this extension is unique and is
necessarily a homomorphism.

We refer the reader to \cite{DR2} for some results concerning
representations $\pi\colon G\to B(X)$ extending to an $R$-bounded
homomorphism $w\colon C^{*}_{\lambda}(G)\to B(X)$ in the case when
$G$ is abelian, and their relationships with $R$-bounded spectral
measures (see also Remark \ref{4Abelian}).

\medskip
\section{Multipliers on the Kalton-Weis $\ell$-spaces.}
We will need abstract Hilbert space valued Banach spaces, usually
called $\ell$-spaces, which were introduced by Kalton and Weis in
the unpublished paper \cite{KaW2}. These $\ell$-spaces allow to
define abstract square functions and were used in \cite{KaW2} to
deal with relationships between $H^{\infty}$ calculus and square
function estimates. Similar spaces are constructed in \cite{JLX}
for the same purpose, in the setting of noncommutative
$L^{p}$-spaces. Recently, $\ell$-spaces played an important role
in the development of vector valued stochastic integration (see in
particular \cite{VNVW1,VNVW2}) and for control theory in a Banach space setting
\cite{HK2}. In this section, we first recall some definitions and
basics of $\ell$-spaces, and then we develop specific properties
which will be useful in the next section.

Let $X$ be a Banach space and let $H$ be a Hilbert space. We let
$\bH$ denote the conjugate space of $H$. We will identify the
algebraic tensor product $\bH\otimes X$ with the subspace of
$B(H,X)$ of all bounded finite rank operators in the usual way.
Namely for any finite families $(\xi_k)_k$ in $H$ and $(x_k)_k$ in
$X$, we identify the element $\sum_k\overline{\xi_k}\otimes x_k$
with the operator $u\colon H\to X$ defined by letting
$u(\eta)=\sum_k\langle\eta,\xi_k\rangle x_k\,$ for any $\eta\in
H$.

For any $u\in \bH\otimes X$, there exists a finite orthonormal
family $(e_k)_k$ of $H$ and a finite family $(x_k)_k$ of $X$ such
that $u=\sum_k \overline{e_k}\otimes x_k$. Then we set
$$
\Gnorm{u}=\,\Bignorm{\sum_k g_k\otimes x_k}_{G(X)}.
$$
Using Lemma \ref{2Inv}, it is easy to check that this definition
does not depend on the $e_k$'s and $x_k$'s representing $u$. Next
for any $u\in B(H,X)$, we set
$$
\lnorm{u}=\,\sup\bigl\{ \Gnorm{uP}\, \big\vert \, P\colon H\to H \
\hbox{finite rank orthogonal projection}\, \bigr\}.
$$

Note that the above quantity may be infinite. Then 
we denote by $\ell_+(H,X)$ the space of all bounded operators
$u\colon H\to X$ such that $\lnorm{u}<\infty$. This is a Banach
space for the norm $\lnorm{\ }$. We let $\ell(H,X)$ denote the
closure of $\bH\otimes X$ in $\ell_+(H,X)$. It is observed in \cite{KaW2}
that $\ell(H,X)=\ell_+(H,X)$ provided that $X$ does not contain
$c_0$ (we will not use this fact in this paper).

\begin{proposition}\label{3Tensor} Let $S\in B(H)$.
\begin{itemize}
\item [(1)] For any finite rank operator $u\colon H\to X$, we have
$\Gnorm{u\circ S}\leq \Gnorm{u}\norm{S}$. \item [(2)] For any
$u\in\ell_+(H,X)$, the operator $u\circ S$ belongs to $\ell_+(H,X)$ and
$\lnorm{uS}\leq\lnorm{u}\norm{S}$.
\end{itemize}
\end{proposition}

\begin{proof} Part (1) is a straightforward consequence of Lemma \ref{2Inv}.
Indeed suppose that $u=\sum_i \overline{e_i}\otimes x_i$ for some
finite orthonormal family $(e_i)_i$ of $H$ and some $x_i\in X$.
Then if $(e'_j)_j$ is an orthonormal basis of ${\rm
Span}\{S^{*}(e_i)\, :\, i=1,\ldots,n\}$, we have
$$
u\circ S=\sum_{i,j}\langle e_i,S(e'_j)\rangle\,\overline{e'_j}\otimes
x_i\,.
$$
Hence
$$
\Gnorm{u\circ S} = \Bignorm{\sum_{i,j}\langle e_i, S(e'_j)\rangle\,
g_j\otimes x_i}_{G(X)}  \leq\bignorm{\langle e_i,
S(e'_j)\rangle}_{\ell^{2}\to\ell^{2}} \Bignorm{\sum_{i} g_i\otimes
x_i}_{G(X)}  \leq \norm{S} \Gnorm{u}.
$$
To prove (2), consider an arbitrary $u\colon H\to X$ and let
$P\colon H\to H$ be a finite rank orthogonal projection. Then $SP$
is finite rank hence there exists a finite rank orthogonal
projection $Q\colon H\to H$ such that $SP=QSP$. Applying the first
part of this proof to $uQ$, we infer that
$$
\Gnorm{uSP}=\Gnorm{uQQSP}\leq\Gnorm{uQ}\norm{QSP}\leq\lnorm{u}\norm{S}.
$$
The result follows by passing to the supremum over $P$.
\end{proof}

\begin{remark}\label{3Rem} \ 

(1) It is clear from above that for any finite rank $u\colon H\to X$, we have $\Gnorm{u}=\lnorm{u}$. 
More generally for any $u\colon H\to X$, we have $\lnorm{u}=\sup\{\Gnorm{uw}\}$, where
the supremum runs over all finite rank operators $w\colon H\to H$
with $\norm{w}\leq 1$.

(2) Let $S\in B(H)$ and let $\varphi_S \colon B(H,X)\to B(H,X)$ be defined by $\varphi_S(u)=u\circ S$. 
It is easy to check (left to the reader) that the  restriction of $\varphi_S$ to
$\bH\otimes X$ coincides with $\overline{S^{*}}\otimes I_{X}$.
\end{remark}

\bigskip
We will now focus on the case when $H=L^{2}(\Omega,\mu)$, for some
arbitrary measure space $(\Omega,\mu)$. We will identify $H$ and
$\bH$ in the usual way. We let $L^{2}(\Omega; X)$ be the
associated Bochner space and we recall that $L^{2}(\Omega)\otimes
X$ is dense in $L^{2}(\Omega; X)$. There is a natural embedding of
$L^{2}(\Omega; X)$ into $B(L^{2}(\Omega),X)$ obtained by
identifying any $F\in L^{2}(\Omega; X)$ with the operator
$$
u_F\colon f\,\longmapsto\,\int_{\Omega} F(t)f(t)\,d\mu(t)\,,\qquad
f\in L^{2}(\Omega).
$$
Thus we have the following diagram of embeddings, that we will use without any further reference.
For example, it will make sense through these identifications to compute
$\lnorm{F}$ for any $F\in L^{2}(\Omega; X)$.

\begin{displaymath}
\xymatrix{ & L^{2}(\Omega;X) \ar[rrd] & &  \\
L^{2}(\Omega)\otimes X  \ar[ru]\ar[rd] & & & B(L^{2}(\Omega),X)\\
& \ell(L^2(\Omega),X) \ar[r] & \ell_+(L^2(\Omega),X) \ar[ru] & }
\end{displaymath}

\medskip
By a subpartition of $\Omega$, we mean a finite
set $\theta =\{I_{1},\ldots,I_{m}\}$ of pairwise disjoint
measurable subsets of $\Omega$ such that $0<\mu(I_i)<\infty$ for
any $i=1,\ldots,m$. We will use the natural partial order on
subpartitions, obtained by saying that $\theta\leq\theta'$ if and
only if each set in $\theta$ is a union of some sets in $\theta'$.
For any subpartition $\theta =\{I_{1},\ldots,I_{m}\}$, we let
$E_\theta\colon L^{2}(\Omega)\to L^{2}(\Omega)$ be the orthogonal
projection defined by
$$
E_\theta(f)\,=\,\sum_{i=1}^{m}\frac{1}{\mu(I_i)}\Bigl(\int_{I_{i}}
f(t)\,d\mu (t)\,\Bigr)\, \chi_{I_{i}},\qquad f\in L^{2}(\Omega).
$$
It is plain that $\lim_{\theta\to\infty}\norm{E_\theta(f)-f}_2=0$
for any $f\in L^{2}(\Omega)$. Now let
$$
E_{\theta}^{X}\colon B(L^{2}(\Omega),X)\longrightarrow
L^{2}(\Omega)\otimes X
$$
be defined by $E_{\theta}^{X}(u)=uE_\theta$. Then the above
approximation property extends as follows.

\begin{lemma}\label{3Approx}\
\begin{itemize}
\item [(1)] For any $u\in\ell(L^{2}(\Omega),X)$,
$\lim_{\theta\to\infty}\lnorm{E_{\theta}^{X}(u) -u}=0$. \item
[(2)] For any $u\in L^{2}(\Omega;X)$,
$\lim_{\theta\to\infty}\norm{E_{\theta}^{X}(u)
-u}_{L^{2}(\Omega;X)}=0$.
\end{itemize}
\end{lemma}

\begin{proof}
By Remark \ref{3Rem}, (2), the restriction of $E_{\theta}^{X}$ to
$L^{2}(\Omega)\otimes X$ coincides with $E_{\theta}\otimes I_X$,
hence (1) holds true if $u\in L^{2}(\Omega)\otimes X$. According
to Proposition \ref{3Tensor}, we have
$$
\bignorm{E_{\theta}^{X}\colon \ell(L^{2}(\Omega),X)\longrightarrow
\ell(L^{2}(\Omega),X)}\leq 1.
$$
Since $L^{2}(\Omega)\otimes X$ is dense in
$\ell(L^{2}(\Omega),X)$, part (1) follows by
equicontinuity. The proof of (2) is identical.
\end{proof}

\begin{lemma}\label{3Partition}
For any $u\in B(L^{2}(\Omega),X)$  and any subpartition
$\theta_{0}$ of $\Omega$,
$$
\lnorm{u}\,=\sup\bigl\{\Gnorm{uE_\theta}\, :\, \theta\ {\rm
subpartition\ of}\  \Omega,\ \theta\geq\theta_{0}\,\bigr\}.
$$
\end{lemma}

\begin{proof}
Let $P\colon L^2(\Omega)\to L^2(\Omega)$ be a finite rank orthogonal projection, and
let $(h_1,\ldots,h_n)$ be an orthonormal basis of its range. Then
$$
uP=\sum_k h_k\otimes u(h_k)\qquad\hbox{and}\qquad u
E_\theta P =\sum_k h_k \otimes uE_\theta (h_k)
$$
for any subpartition $\theta$. Since $E_\theta(h_k)\to h_k$ for
any $k=1,\ldots,n$, we deduce that
$$
\Gnorm{uP}\,=\,\lim_{\theta\to\infty}\Gnorm{uE_\theta P}.
$$
By Proposition \ref{3Tensor}, this implies that $\Gnorm{uP}\leq
\sup_{\theta\geq\theta_0}\Gnorm{uE_\theta}$ and the result follows at once.
\end{proof}

Let $\phi\colon \Omega\to B(X)$ be a bounded strongly measurable
function. We may define a multiplication operator $T_\phi\colon
L^{2}(\Omega;X)\to L^{2}(\Omega;X)$ by letting
$$
\bigl[T_\phi(F)\bigr](t)= \phi(t)F(t),\qquad F\in L^{2}(\Omega;X).
$$
Consider the associated bounded set
\begin{equation}\label{3F}
F_\phi \,=\,\biggl\{\frac{1}{\mu(I)}\,\int_{I}\phi(t)\, d\mu(t)\
:\, I\subset \Omega,\ 0<\mu(I)<\infty\,\biggr\}.
\end{equation}
The following is an analog of \cite[Prop. 4.4]{JLX} and extends \cite[Prop. 4.11]{KaW2}.

\begin{proposition}\label{3Mult} If the set $F_\phi$ is $\gamma$-bounded,
there exists a (necessarily unique) bounded operator
$$
M_\phi\colon
\ell(L^{2}(\Omega),X)\longrightarrow\ell_{+}(L^{2}(\Omega),X),
$$
such that $M_\phi$ and $T_\phi$ coincide on the intersection
$\ell(L^{2}(\Omega),X)\cap L^{2}(\Omega;X)$. Moreover we have
$$
\norm{M_\phi}\leq\gamma(F_\phi).
$$
\end{proposition}

\begin{proof}
Let $\E\subset L^{2}(\Omega)$ be the dense subspace of all simple
functions and let $u\in\E\otimes X$. There exists a subpartition
$\theta_{0} =(A_1,\ldots, A_N)$ and some $x_1,\ldots, x_N$ in $X$
such that 
$$
u=\sum_j \chi_{A_j}\otimes x_j\,. 
$$
Let
$\theta=(I_{1},\ldots, I_{m})$ be another subpartition and assume
that $\theta_{0}\leq\theta$. Thus there exist
$\alpha_{ij}\in\{0,1\}$ for $i=1,\ldots,m$ and $j=1,\ldots,N$ such
that $\chi_{A_{j}}=\sum_i\alpha_{ij}\chi_{I_{i}}\,$ for any $j$.
Consequently, we have
$$
u=\sum_{i,j}\alpha_{ij}\chi_{I_{i}}\otimes
x_j\qquad\hbox{and}\qquad \bigl[T_\phi(u)\bigr](t)\,=\,
\sum_{i,j}\alpha_{ij}\chi_{I_{i}}(t)\phi(t) x_j.
$$
For any $i=1,\ldots,m$, let
$$
T_i=\,\frac{1}{\mu(I_i)}\,\int_{I_i}\phi(t)\, d\mu(t)\,.
$$
Then a thorough look at the definition of $E_\theta^X$ shows that
$$
E_\theta^X\bigl(T_\phi(u))=\,\sum_{i,j}\alpha_{ij}\,\chi_{I_{i}} \otimes
T_i(x_j).
$$
Since $\bigl(\mu(I_i)^{-\frac{1}{2}}\chi_{I_{i}}\bigr)_i$ is an
orthonormal family of $L^{2}(\Omega)$, this implies that
$$
\bigGnorm{E_\theta^X\bigl(T_\phi(u))}\,=
\,\Bignorm{\sum_{i,j}\alpha_{ij}\mu(I_i)^{\frac{1}{2}} g_i\otimes
T_i(x_j)}_{G(X)}.
$$
Likewise,
$$
\Gnorm{u}= \Bignorm{\sum_{i,j}\alpha_{ij}\mu(I_i)^{\frac{1}{2}}
g_i\otimes x_j}_{G(X)}.
$$
Since each $T_i$ belongs to the set $F_\phi$, this implies that
$\bigGnorm{E_\theta^X\bigl(T_\phi(u))}\leq\,\gamma(F_\phi)\Gnorm{u}$.
Taking the supremum over $\theta$ and applying Lemma
\ref{3Partition}, we obtain that
$T_\phi(u)\in\ell_{+}(L^{2}(\Omega),X)$, with
$$
\lnorm{T_\phi(u)}\leq \gamma(F_\phi)\Gnorm{u}.
$$
This induces a bounded operator $M_\phi\colon
\ell(L^{2}(\Omega),X)\to \ell_{+}(L^{2}(\Omega),X)$ coinciding
with $T_\phi$ on $\E\otimes X$ and verifying $\norm{M_\phi}\leq
\gamma(  F_\phi)$.

To show that $M_\phi$ and $T_\phi$ coincide on
$\ell(L^{2}(\Omega),X)\cap L^{2}(\Omega;X)$, let $u$ belong to
this intersection and note that by construction,
$M_\phi(E_\theta^X(u))=T_\phi(E_\theta^X(u))$ for any subpartition
$\theta$. Then the equality $M_\phi(u)=T_\phi(u)$ follows from
Lemma \ref{3Approx}.
\end{proof}

\bigskip
In the rest of this section, we consider natural tensor extensions
of the spaces and multipliers considered so far. Let $N\geq 1$ be
a fixed integer and let $(e_1,\ldots,e_N)$ denote the canonical
basis of $\ell^{2}_N$. We let $\ell^{2}_N\h L^{2}(\Omega)$ be the
Hilbert space tensor product of $\ell^{2}_N$ and $L^{2}(\Omega)$.
For any bounded operator $u\colon\ell^{2}_N\h L^{2}(\Omega)\to X$
and any $k=1,\ldots, N$, let $u_k\colon L^{2}(\Omega)\to X$ be
defined by $u_k(f)= u(e_k\otimes f)$. Then the mapping $u\mapsto
\sum_k e_k\otimes u_k$ induces an algebraic isomorphism
\begin{equation}\label{3Ident}
B\bigl(\ell^{2}_N \h L^{2}(\Omega),X\bigr)\,\simeq\,\ell^{2}_N
\otimes B(L^{2}(\Omega),X).
\end{equation}
Let us now see the effects of this isomorphism on the special spaces considered so far.
Let $\Omega_N=\Omega\times\{1,\ldots,N\}$, so that we have a natural isometric isomorphism
$$
\ell^{2}_N\h
L^{2}(\Omega)=L^{2}(\Omega_N).
$$
Then it is clear that under the identification
(\ref{3Ident}), an operator $u\colon L^{2}(\Omega_N)\to X$ belongs
to $L^{2}(\Omega_N;X)$ if and only if $u_k$ belongs to $L^{2}(\Omega;X)$ for
any $k=1,\ldots, N$. Moreover this induces an isometric isomorphism
identification
$$
L^{2}(\Omega_N;X)\,=\,\ell^{2}_N\bigl(L^{2}(\Omega;X)\bigr).
$$
Likewise it is easy to check (left to the reader) that $u\colon L^{2}(\Omega_N)\to X$ belongs
to $\ell_{+}(L^{2}(\Omega_N),X)$ (resp. $\ell(L^{2}(\Omega_N),X)$)
if and only if $u_k$ belongs to $\ell_{+}(L^{2}(\Omega),X)$
(resp. $\ell(L^{2}(\Omega),X)$) for any $k=1,\ldots, N$, which
leads to algebraic isomorphisms
\begin{equation}\label{3ell}
\ell_{+}(L^{2}(\Omega_N),X)\,\simeq\,\ell^{2}_N\otimes
\ell_{+}(L^{2}(\Omega),X)\quad\hbox{and}\quad
\ell(L^{2}(\Omega_N),X)\,\simeq\,\ell^{2}_N\otimes
\ell(L^{2}(\Omega),X).
\end{equation}
Now let $\phi\colon\Omega\to B(X)$ be a bounded strongly measurable function as before and let
$\phi_N\colon \Omega_N\to B(X)$ be defined by
$$
\phi_N(t,k)=\phi(t),\qquad t\in \Omega,\ k=1,\ldots,N.
$$
As in (\ref{3F}), we may associate a set $F_{\phi_N}\subset B(X)$ to
$\phi_N$. A moment's thought shows that  $F_\phi\subset F_{\Phi_N}\subset{\rm co}(F_\phi)$. Hence   $F_{\phi_N}$ is $\gamma$-bounded if and
only if $F_{\phi}$ is $\gamma$-bounded and we have
$$
\gamma\bigl(F_{\phi_N}\bigr)= \gamma\bigl(F_\phi\bigr)
$$
in this case. It is clear that under the identifications (\ref{3ell}), 
the associated multiplier operator $M_{\phi_N}\colon \ell(L^{2}(\Omega_N),X)\to \ell_{+}(L^{2}(\Omega_N),X)$ satisfies
\begin{equation}\label{3M}
M_{\phi_N} = I_{\ell^{2}_N}\otimes M_\phi.
\end{equation}

\medskip
\section{Characterization of $\gamma$-bounded representations of amenable groups}
Thoughout we let $G$ be a locally compact group equipped with a
left Haar measure and for any measurable $I\subset G$, we simply
let $\vert I\vert$ denote the measure of $I$. If $\pi\colon G\to
B(X)$ is any bounded representation and $\norm{\pi}=\sup_{t\in
G}\norm{\pi(t)}$, it is plain that for any $I\subset G$ and any
$z\in X$, we have 
$$
\norm{\pi}^{-1}\vert I \vert^{\frac{1}{2}}\norm{z}\leq
\bigl(\int_{I}\norm{\pi(t)z}^{2}\, dt\,\bigr)^{\frac{1}{2}}\leq
\norm{\pi}\vert I \vert^{\frac{1}{2}}\norm{z}.
$$
The first part of the
following lemma is an analog of this double estimate when the
space $L^{2}(G;X)$ is replaced by $\ell_{+}(L^{2}(G),X)$. In the second part,
we apply the principles explained at the end of the previous section.

\begin{lemma}\label{4MainLemma} Let $\pi\colon G\to B(X)$ be a 
$\gamma$-bounded representation and
let $I\subset G$ be any measurable subset of $G$ with finite measure.
\begin{itemize}
\item [(1)] For any $z\in X$, the function $t\mapsto
\chi_{I}(t)\pi(t)z$ belongs to $\ell_{+}(L^{2}(G),X)$ and we have
$$
\gamma(\pi)^{-1}\vert I
\vert^{\frac{1}{2}}\norm{z}\leq\,\biglnorm{t\mapsto
\chi_{I}(t)\pi(t)z}\,\leq \gamma(\pi) \vert I
\vert^{\frac{1}{2}}\norm{z}.
$$
\item [(2)] Let $N\geq 1$ be an integer. Let $z_1,\ldots, z_N\in
X$ and let $F_k(t)= \chi_{I}(t)\pi(t)z_k$ for any $k=1,\ldots,N$.
Then
$$
\gamma(\pi)^{-1}\vert I \vert^{\frac{1}{2}}\BigGnorm{\sum_k
e_k\otimes z_k}\leq\,\Biglnorm{\sum_k e_k\otimes F_k}\,\leq
\gamma(\pi) \vert I \vert^{\frac{1}{2}}\BigGnorm{\sum_k e_k\otimes
z_k}.
$$
\end{itemize}
\end{lemma}

\begin{proof} Part (1) is a special case of part (2) so we only need to
prove the second statement. The upper estimate 
is a simple consequence of Proposition \ref{3Mult} applied with $\pi=\phi$, 
and the discussion at the end of Section 3. Indeed, let
$F_{\pi}$ be the set associated with $\pi\colon G\to B(X)$ as in
(\ref{3F}). For any $I\subset G$ with $0<\vert I\vert<\infty$, the
operator $\vert I\vert^{-1}\int_{I}\pi(t)\, dt\,$ belongs to the
strong closure of the absolute convex hull of $\{\pi(t)\,:\, t\in
G\}$. Hence $\gamma(F_\pi)\leq\gamma(\pi)$ by Lemma \ref{2Conv}.
Let
$$
M_\pi\colon \ell(L^{2}(G),X)\longrightarrow \ell_{+}(L^{2}(G),X)
$$
be the multiplier operator associated with $\pi$. Then for any
$I\subset G$ and any $z\in X$, the function $t\mapsto
\chi_{I}(t)\pi(t)z$ is equal to $M_\pi(\chi_{I}\otimes z)$. Thus
according to (\ref{3M}), we have
$$
\sum_k e_k\otimes F_k \,=\, M_{\pi_{n}}\Bigl(\sum_k
e_k\otimes\chi_{I}\otimes z_k\Bigr).
$$
Moreover $\bigl(\vert I\vert^{-\frac{1}{2}}
e_k\otimes\chi_{I}\bigr)_k$ is an orthonormal family of
$\ell^{2}_{N}\h L^{2}(G)$, hence
$$
\Biglnorm{\sum_k e_k\otimes\chi_{I}\otimes z_k}\,=\,\vert I
\vert^{\frac{1}{2}}\BigGnorm{\sum_k e_k\otimes z_k}.
$$
Consequently  we have
$$
\Biglnorm{\sum_k e_k\otimes
F_k}\leq\gamma\bigl(F_{\pi_N}\bigr)\Biglnorm{\sum_k
e_k\otimes\chi_{I}\otimes z_k}\leq\gamma(\pi)\vert I
\vert^{\frac{1}{2}}\BigGnorm{\sum_k e_k\otimes z_k}.
$$

We now turn to the lower estimate, for which we will use duality.
For any $\varphi_{1},\ldots,\varphi_{N}$ in $X^{*}$, we set
$$
\bignorm{(\varphi_{1},\ldots,\varphi_{N})}_{\ell^{*}}\,=
\,\sup\biggl\{\Bigl\vert\sum_{k=1}^{N}\,\langle\varphi_{k},x_k\rangle\,\Bigr\vert\,
:\, x_1,\ldots, x_N\in X,\ \Bignorm{\sum_{k=1}^{N} g_k\otimes
x_k}_{G(X)}\leq 1\,\biggr\}.
$$
We fix some $I\subset G$ with $0<\vert I\vert <\infty$. Then we
consider $z_1,\ldots,z_N$ in $X$ and the functions $F_1,\ldots, F_N$ in
$L^{2}(G;X)$ given by $F_k(t) =\chi_{I}(t)\pi(t)z_k$. By
Hahn-Banach there exist $\varphi_{1},\ldots,\varphi_{N}$ in
$X^{*}$ such that
$$
\bignorm{(\varphi_{1},\ldots,\varphi_{N})}_{\ell^{*}}=1\qquad\hbox{and}\qquad
\BigGnorm{\sum_k e_k\otimes
z_k}=\,\sum_k\langle\varphi_{k},z_k\rangle\,.
$$
Using the latter equality and  Lemma \ref{3Approx}, (2), we thus
have
\begin{align*}
\vert I\vert\,\BigGnorm{\sum_k e_k\otimes z_k}& =\,\sum_k \int_{I}
\langle\varphi_k,z_k\rangle\,dt\\ & =\,\sum_k\,\int_{I}\langle
\pi(t^{-1})^{*}\varphi_k,\pi(t)z_k\rangle\, dt\\
& =\,\sum_k \int_G\bigl\langle\chi_{I}(t)\pi(t^{-1})^{*}\varphi_k,F_k(t)\bigr\rangle\, dt\\
& =\,\lim_{\theta\to\infty}
\,\sum_k\,\int_G\bigl\langle\chi_{I}(t)\pi(t^{-1})^{*}\varphi,
\bigl[E_\theta^{X}(F_k)\bigr](t)\bigr\rangle\, dt\,.
\end{align*}
Let $\theta=(I_{1},\ldots, I_{m})$ be a subpartition of $G$ such
that $I=I_{1}\cup\cdots\cup I_{n}$ for some $n\leq m$ and let
$$
J_\theta= \,\sum_k\,
\int_G\bigl\langle\chi_{I}(t)\pi(t^{-1})^{*}\varphi_k,
\bigl[E_\theta^{X}(F_k)\bigr](t)\bigr\rangle\, dt
$$
be the above sum of integrals. For any $i=1,\ldots,n$, let
$$
T_i=\,\frac{1}{\vert I_i\vert}\,\int_{I_i}\pi(t)\,
dt\qquad\hbox{and}\qquad S_i=\, \frac{1}{\vert
I_i\vert}\,\int_{I_i}\pi(t^{-1})\, dt\,.
$$
For any $k$ we have
\begin{equation}\label{4E}
E_\theta^{X}(F_k)=\,\sum_{i=1}^{n} \chi_{I_i}\otimes T_i(z_k)\,.
\end{equation}
We deduce that
\begin{align*}
J_\theta & =\,\sum_{k=1}^{N} \sum_{i=1}^{n}\,\int_{I_i}\langle
\pi(t^{-1})^{*}\varphi_k,  T_i(z_k) \rangle\, dt\\
& =\, \sum_{k=1}^{N} \sum_{i=1}^{n}\, \int_{I_i}\langle \varphi_k,
\pi(t^{-1}) T_i(z_k) \rangle \, dt\\
& =\,\sum_{k=1}^{N}  \sum_{i=1}^{n}\, \vert I_i \vert \, \langle
\varphi_k , S_i T_i(z_k)\rangle.
\end{align*}
According to the definition of the $\ell^{*}$-norm, this identity
implies that
$$
\vert J_\theta\vert\leq\,\Bignorm{\sum_k
g_k\otimes\Bigl(\sum_i\vert I_{i}\vert S_iT_i(z_k)\Bigr)}_{G(X)}.
$$
Let $a\colon\ell^{2}_{nN}\to \ell^{2}_{N}$ be defined by
$$
a\Bigl((c_{ik})_{\substack{1\leq i\leq n \\ 1\leq  k\leq N}}\Bigr)\,=\,\Bigl(\sum_i
c_{ik}\vert I_i\vert^{\frac{1}{2}}\Bigr)_{k},\qquad c_{ik}\in\Cdb.
$$
Let $c=(c_{ik})$ in $\ell^{2}_{nN}$. Using Cauchy-Schwarz and the
fact that $\vert I\vert=\sum_i\vert I_i\vert$, we have
\begin{align*}
\norm{a(c)}_{2}^{2} & = \,\sum_k\Bigl\vert\sum_i c_{ik}
\vert I_i\vert^{\frac{1}{2}}\Bigr\vert^{2}\\
&\leq \sum_k\Bigl(\sum_i\vert c_{ik}\vert^{2}\Bigr)\,\Bigl( \sum_i
\vert I_i\vert\Bigr)\, =\vert I\vert \norm{c}_{2}.
\end{align*}
Hence $\norm{a}\leq \vert I\vert^{\frac{1}{2}}$. Let $(g_{ik})_{i,k\geq 1}$ be a doubly indexed 
family of  independent standard Gaussian variables. According to Lemma
\ref{2Inv}, the latter estimate implies that
$$
\Bignorm{\sum_{i,k} g_k\otimes \vert I_i\vert^{\frac{1}{2}}
y_{ik}}_{G(X)}\,\leq\, \vert I\vert^{\frac{1}{2}}\,
\Bignorm{\sum_{i,k} g_{ik}\otimes y_{ik}}_{G(X)}
$$
for any $y_{ik}$ in $X$. We deduce that
$$
\vert J_\theta\vert\leq\,\vert I\vert^{\frac{1}{2}}\,
\Bignorm{\sum_{i,k} g_{ik} \otimes \vert I_i\vert^{\frac{1}{2}}
S_i T_i (z_k)}_{G(X)}.
$$
Next observe that by convexity again, we
have $\gamma\bigl(\{S_1,\ldots,S_n\}\bigr)\leq\gamma(\pi)$. The
latter estimate therefore implies that
$$
\vert J_\theta\vert\leq\,\gamma(\pi)\vert I\vert^{\frac{1}{2}}\,
\Bignorm{\sum_{i,k} g_{ik} \otimes \vert I_i\vert^{\frac{1}{2}}
T_i (z_k)}_{G(X)}.
$$
Since $(\vert I_i\vert^{-\frac{1}{2}} \chi_{I_i})_i$ is an orthonormal family
of $L^{2}(G)$, we have, using (\ref{4E}),
\begin{align*}
\Bignorm{\sum_{i,k} g_{ik} \otimes \vert I_i \vert^{\frac{1}{2}}
T_i (z_k)}_{G(X)} & = \Bignorm{\sum_{i,k} e_k\otimes \chi_{I_i}
\otimes T_i (z_k)}_{G}\\
&= \Bignorm{\bigl(I_{\ell^{2}_N}\otimes
E_\theta^{X}\bigr)\Bigl(\sum_k e_k\otimes F_k\Bigr)}_{G}\\ &\leq
\Biglnorm{\sum_k e_k\otimes F_k}.
\end{align*}
Hence
$$
\vert J_\theta\vert\leq\,\gamma(\pi)\vert I\vert^{\frac{1}{2}}\,
\Biglnorm{\sum_k e_k\otimes F_k},
$$
and passing to the limit when $\theta\to\infty$, this yields the
lower estimate.
\end{proof}

\begin{remark}\label{4op}
The above lemma remains true if $\pi(t)$ is replaced by $\pi(t^{-1})$. 
This follows either from the proof itself, or by considering the representation 
$\pi^{\rm op}\colon G^{\rm op}\to B(X)$ defined by $\pi^{\rm op}(t)=\pi(t^{-1})$.
Here $G^{\rm op}$ denotes the opposite group of $G$, i.e. 
$G$ equipped with the reverse product. 
\end{remark}

The following notion was introduced in \cite{KL}. 
For any $C^*$-algebra $A$, the space $M_N(A)$ of $N\times N$ 
matrices with entries in $A$ is equipped with its unique $C^*$-norm.

\begin{definition}\label{4Mat}
Let $A$ be a $C^{*}$-algebra and let $w\colon A\to B(X)$ be a bounded linear map. 
\begin{itemize}
\item [(1)] We say that $w$ is matricially $\gamma$-bounded if there is a constant $C\geq 0$ such that
\begin{equation}\label{4Mat-gamma}
\Bignorm{\sum_{i,j=1}^{N} g_i\otimes w(a_{ij})x_j}_{G(X)}\,\leq\,C\, \bignorm{[a_{ij}]}_{M_N(A)}\,
\Bignorm{\sum_{j=1}^{N} g_j\otimes  x_j}_{G(X)}
\end{equation}
for any $N\geq 1$, for any $[a_{ij}]\in M_N(A)$ and for any $x_1,\ldots, x_N\in X$. In this case we let 
$\matnorm{w}$ denote the smallest possible $C$. 
\item [(2)] We say that $w$ is matricially $R$-bounded if (\ref{4Mat-gamma}) 
holds when the Gaussian sequence $(g_k)_k$ is replaced by a Rademacher sequence $(\varepsilon_k)_k$, and we let 
$\matnormR{w}$ denote the smallest possible constant in this case.
\end{itemize}
\end{definition}

Two simple comments are in order (see \cite[Remark 4.2]{KL} for details). First, 
restricting  (\ref{4Mat-gamma}) to the case when $[a_{ij}]$ is a diagonal matrix, 
we obtain that any matricially $\gamma$-bounded map $w\colon A\to B(X)$ is $\gamma$-bounded,
with 
$$
\gamma(w)\leq\matnorm{w}.
$$
Second, if $X=H$ is a Hilbert space, 
then $\gamma$-matricial boundedness coincides
with complete boundedness and  we have
$\matnorm{w}=\cbnorm{w}$ (the completely bounded norm of $w$).
Similar comments apply to $R$-boundedness.

The proof of our main result below uses transference techniques from \cite{CW} in the 
framework of $\ell$-spaces.

\begin{theorem}\label{4MainThm} Let $G$ be an amenable locally compact group and let
$\pi\colon G\to B(X)$ be a bounded representation. The following
assertions are equivalent.
\begin{itemize}
\item [(i)] $\pi$ is $\gamma$-bounded. 
\item [(ii)] $\pi$ extends
to a bounded homomorphism $w\colon C^{*}_{\lambda}(G)\to B(X)$ 
(in the sense of Definition \ref{2Def}) and $w$ is $\gamma$-bounded.
\end{itemize}
In this case, $w$ is matricially $\gamma$-bounded and
$$
\gamma(\pi)\leq\gamma(w)\leq\matnorm{w}\leq\gamma(\pi)^{2}.
$$
\end{theorem}

\begin{proof}
Assume (ii) and let $\sigma_\pi\colon L^{1}(G)\to B(X)$ be induced
by $\pi$. Then $\sigma_\pi = w\circ \sigma_\lambda$ and
$\sigma_\lambda$ is a contraction. Hence $\sigma_\pi$ is
$\gamma$-bounded, with $\gamma(\sigma_\pi)\leq\gamma(u)$. Then (i)
follows from Lemma \ref{2Equiv} and we have $\gamma(\pi)\leq\gamma(w)$.

\smallskip
Assume (i). Our proof of (ii) will be divided into two parts. We
first show that for any $k\in L^{1}(G)$, we have
\begin{equation}\label{4w}
\norm{\sigma_\pi(k)}\,\leq\gamma(\pi)^{2}
\norm{\sigma_\lambda(k)}.
\end{equation}
This implies the existence of $w\colon C^{*}_{\lambda}(G)\to B(X)$
extending $\pi$. Then we will show (\ref{4w-mat}), which implies
that $w$ is actually $\gamma$-bounded. Although (\ref{4w}) is a
special case of (\ref{4w-mat}), establishing that estimate first
makes the proof easier to read.

\smallskip
Let $k\in L^{1}(G)$ and assume that $k$ has a compact support
$\Gamma\subset G$. Let $V\subset G$ be an arbitrary open
neighborhood of the unit $e$, with $0<\vert V\vert<\infty$. We let
$T\colon L^{2}(G)\to L^{2}(G)$ be the multiplication operator
defined by letting $T(f)=\chi_{V}f$ for any $f\in L^{2}(G)$. Then
we let $S\colon L^{2}(G)\to L^{2}(G)$ be defined by
$$
\bigl(Sg\bigr)(s)=\,\int_{G} k(t)g(ts)\, dt,\qquad g\in L^{2}(G),\
s\in G.
$$
Under the natural duality between $L^{2}(G)$ and itself, $S$ is
the transposed map of $\sigma_\lambda(k)$, hence
\begin{equation}\label{4S}
\norm{S}=\norm{\sigma_\lambda(k)}.
\end{equation}

Let $x\in X$. The set $\Gamma^{-1}V\subset G$ has a positive and
finite measure, hence applying Lemma \ref{4MainLemma} (and Remark \ref{4op}), 
we see that the function 
$$
F\colon s\longmapsto
\chi_{\Gamma^{-1}V}(s)\pi(s^{-1})x
$$
belongs to $L^{2}(G;X)\cap
\ell_{+}(L^{2}(G),X)$. Let $u\colon L^{2}(G)\to X$ be the bounded
operator associated to $F$ and let $\widetilde{u}=u\circ S\circ
T\in B(L^{2}(G),X)$. Consider an arbitrary $f\in L^{2}(G)$. For
any $h\in L^{2}(G)$,
$$
u(h)=\int_G h(s)\chi_{\Gamma^{-1}V}(s)\pi(s^{-1})x\, ds\,,
$$
hence according to the definitions of $T$ and $S$, we have
$$
\widetilde{u}(f) = \,\int_G\Bigl(\int_G k(t)\chi_V(ts) f(ts)\,
dt\,\Bigr)\, \chi_{\Gamma^{-1}V}(s)\pi(s^{-1})x\, ds\,.
$$
Using Fubini (which is applicable because $\chi_V f$ is
integrable) and the left invariance of $ds$, this implies
\begin{align*}
\widetilde{u}(f) & = \int_G k(t) \Bigl(\int_G \chi_V(ts) f(ts)
\chi_{\Gamma^{-1}V}(s)\pi(s^{-1})x\, ds\,\Bigr)\, dt\\
& = \int_G k(t) \Bigl(\int_G \chi_V(s) f(s)
\chi_{\Gamma^{-1}V}(t^{-1}s)\pi(s^{-1}t)x\, ds\,\Bigr)\, dt\\
& = \int_G \chi_V(s) f(s)\Bigl(\int_G k(t)
\chi_{\Gamma^{-1}V}(t^{-1}s)\pi(s^{-1}t)x\, dt\,\Bigr)\, ds.
\end{align*}
Since $k$ is supported in $\Gamma$ we deduce that
\begin{equation}\label{4v}
\widetilde{u}(f) = \int_G \chi_V(s) f(s)\Bigl(\int_G k(t)
\pi(s^{-1}t)x\, dt\,\Bigr)\, ds\,.
\end{equation}
Let $y=\sigma_\pi(k)x$. For any $s\in G$, we have
$$
\int_G k(t) \pi(s^{-1}t)x\, dt\,= \int_G k(t) \pi(s^{-1})\pi(t)x\,
dt\, =\pi(s^{-1})y.
$$
Thus (\ref{4v}) shows that $\widetilde{u}$ is the bounded operator
associated to the function 
$$
\widetilde{F}\colon s\longmapsto
\chi_V(s)\pi(s^{-1})y. 
$$
By Proposition \ref{3Tensor}, we have
$\lnorm{\widetilde{F}}\leq\norm{ST}\lnorm{F}$. Applying (\ref{4S})
and the fact that $T$ is a contraction, we therefore obtain that
$$
\biglnorm{s\mapsto\chi_V(s)\pi(s^{-1})y}\leq\norm{\sigma_\lambda(k)}
\biglnorm{s\mapsto\chi_{\Gamma^{-1}V}(s)\pi(s^{-1})x}.
$$
Applying Lemma \ref{4MainLemma} (and Remark \ref{4op}) twice we deduce
that
$$
\vert V \vert^{\frac{1}{2}}\norm{y}\leq \gamma(\pi)^{2}
\norm{\sigma_\lambda(k)}\vert
\Gamma^{-1}V\vert^{\frac{1}{2}}\norm{x},
$$
and hence
$$
\norm{\sigma_\pi(k)x}\leq \gamma(\pi)^{2}\Bigl(\frac{\vert
\Gamma^{-1}V\vert}{\vert V\vert}\Bigr)^{\frac{1}{2}}\,
\norm{\sigma_\lambda(k)} \norm{x}.
$$
We now apply the assumption that $G$ is amenable. According to
Folner's condition (see e.g. \cite[Chap. 2]{CW}), we can choose $V$ such that
$\frac{\vert \Gamma^{-1}V\vert}{\vert V\vert}$ is arbitrarily close
to $1$. This yields (\ref{4w}) when $k$ is compactly supported.
Since $\sigma_\lambda$ and $\sigma_\pi$ are continuous, this actually 
implies (\ref{4w}) for any $k\in L^{1}(G)$.

\smallskip
We now aim at showing that $w\colon C_\lambda^{*}(G)\to B(X)$ is
matricially $\gamma$-bounded  and that
$\matnorm{w}\leq\gamma(\pi)^{2}$. In fact the argument is essentially a repetition 
of the above one, modulo standard matrix manipulations. We fix some integer 
$N\geq 1$ and consider 
$x_1,\ldots, x_N$ in $X$. According to Definition \ref{4Mat}, it suffices to show
that for any $[k_{ij}]\in M_N\otimes L^{1}(G)$, we have
\begin{equation}\label{4w-mat}
\Bignorm{\sum_{i,j} g_i\otimes \sigma_{\pi}(k_{ij})x_j}_{G(X)}\,\leq \,\gamma(\pi)^{2}\bignorm{[\sigma_\lambda (k_{ij})]}_{M_N(C^{*}_{\lambda}(G))}\,
\Bignorm{\sum_j g_j\otimes x_j}_{G(X)}.
\end{equation}
In the sequel we let 
$$
x=\sum_{j=1}^{N} e_j\otimes x_j\,\in \ell^{2}_N\otimes X.
$$
Let us identify $M_N\otimes L^{1}(G)$ with $L^{1}(G;M_N)$ in the
natural way and let $k\in L^{1}(G;M_N)$ be the $M_N$-valued function corresponding to $[k_{ij}]$. 
Then
$$
\bigl(I_{M_N}\otimes
\sigma_\pi\bigr)\bigl([k_{ij}]\bigr) \,= \int_G\bigl(k(t)\otimes \pi(t)\bigr)\,dt\qquad\hbox{in}\quad M_N\otimes B(X).
$$
Using the isometric identification
\begin{equation}\label{4Hilbert}
\ell^{2}_N\h L^{2}(G)\,=\, L^{2}(G;\ell^{2}_N),
\end{equation}
we can regard $M_N(C_{\lambda}^{*}(G))$ as a $C^{*}$-subalgebra of $B(L^{2}(G;\ell^{2}_N))$. In this 
situation, it is easy to check that the matrix $[\sigma_
\lambda(k_{ij})]$ corresponds to the operator valued convolution $g\mapsto k*g$ defined by 
$$
(k*g)(s)\,=\,\int_{G} k(t)\bigl[g(t^{-1}s)\bigr]\,dt\,,\qquad 
g\in L^{2}(G;\ell^{2}_N),\ s\in G.
$$
Thus showing (\ref{4w-mat}) amounts to show that
\begin{equation}\label{4w-mat-2}
\Bignorm{\int_G\bigl(k(t)\otimes \pi(t)\bigr)x\,
dt\,}_{G}\leq\gamma(\pi)^{2} \,
\bignorm{k\,*\,\cdotp\colon L^{2}(G;\ell^{2}_N)\longrightarrow L^{2}(G;\ell^{2}_N)}\,\norm{x}_G.
\end{equation}
As in the first part of the proof, we may and do assume that $k$
has a compact support, which we denote by $\Gamma$, and we fix an
arbitrary open neighborhood $V\subset G$ of $e$, with $0<\vert
V\vert<\infty$. 
We let $\widehat{T}=I_{\ell^{2}_N}\otimes T\colon
L^{2}(G;\ell^{2}_N)\to L^{2}(G;\ell^{2}_N)$ be the multiplication
operator by $\chi_V$ and we let $\widehat{S}\colon
L^{2}(G;\ell^{2}_N)\to L^{2}(G;\ell^{2}_N)$ be the transposed map of 
$g\mapsto k*g$.
Let $y_1,\ldots,y_N$ in $X$ such that 
$$
\int_G\bigl(k(t)\otimes \pi(t)\bigr)x\, dt\ =\,\sum_{k=1}^{N} e_k\otimes y_k.
$$
Next for any $k=1,\ldots,N$, let
$$
F_k(s)=\chi_{\Gamma^{-1}V}(s) \pi(s^{-1})x_k
\qquad\hbox{and}\qquad
\widetilde{F_k}(s)=\chi_{V}(s)\pi(s^{-1})y_k.
$$
Then the argument in the first part of this proof and the
identification (\ref{4Hilbert}) show that
$$
\Biglnorm{\sum_k e_k\otimes \widetilde{F_k}}
\leq\norm{\widehat{S}\widehat{T}} \Biglnorm{\sum_k e_k\otimes F_k},
$$
and hence 
$$
\Biglnorm{\sum_k e_k\otimes \widetilde{F_k}}
\leq\,\bignorm{k\,*\,\cdotp\colon L^{2}(G;\ell^{2}_N)\longrightarrow L^{2}(G;\ell^{2}_N)}\,
\Biglnorm{\sum_k e_k\otimes F_k}.
$$
Now using Lemma \ref{4MainLemma}, (2) and arguing as in the first part of the
proof, we deduce (\ref{4w-mat-2}).
\end{proof}

\begin{remark}\label{4Abelian}
If $G$ is an abelian group and $\widehat{G}$ denotes its dual group, 
then the Fourier transform  yields  a natural identification 
$C^{*}_\lambda(G)=C_{0}(\widehat{G})$. Since abelian groups are amenable,
Theorem \ref{4MainThm} provides a 1-1 correspondence between $\gamma$-bounded 
representations $G\to B(X)$ and $\gamma$-bounded nondegenerate 
homomorphisms $C_{0}(\widehat{G})\to B(X)$. 

It is shown in \cite[Prop. 2.2]{DR2} (see also \cite{DR1})
that any $\gamma$-bounded nondegenerate 
homomorphism $w\colon C_{0}(\widehat{G})\to B(X)$ is of the form 
\begin{equation}\label{4Spectral}
w(h) =\,\int_{\widehat{G}} h\, dP,\qquad h\in C_{0}(\widehat{G}),
\end{equation}
where $P$ is a regular strong operator $\sigma$-additive spectral measure 
from the $\sigma$-algebra $\B(\widehat{G})$ of Borel subsets of $\widehat{G}$ into $B(X)$.
Moreover the range of this spectral measure is $\gamma$-bounded. 
Conversely, for any such spectral measure, (\ref{4Spectral}) defines a  
$\gamma$-bounded nondegenerate 
homomorphism $w\colon C_{0}(\widehat{G})\to B(X)$. (In \cite{DR1, DR2}, 
the authors consider $R$-boundedness only but their results hold as well for $\gamma$-boundedness.)

Hence we obtain a 
1-1 correspondence between $\gamma$-bounded 
representations $G\to B(X)$ and regular, $\gamma$-bounded, 
strong operator $\sigma$-additive spectral measures $\B(\widehat{G})\to B(X)$.
\end{remark}

\begin{remark}\label{4HilbertCase}\ 

(1) The above theorem should be regarded as a Banach space version of the Day-Dixmier
unitarization Theorem which asserts that any bounded representation of an amenable 
group $G$ on some Hilbert space $H$ is unitarizable (see \cite[Chap. 0]{Pis2}). Indeed when $X=H$, the
main implication `(i)$\,\Rightarrow\,$(ii)' of Theorem \ref{4MainThm} says that 
any bounded representation $\pi\colon G\to B(H)$ extends to a completely bounded homomorphism
$w\colon C_{\lambda}^{*}(G)\to B(H)$, with $\cbnorm{w}\leq\norm{\pi}^{2}$. According to Haagerup's 
similarity Theorem \cite{Ha}, this implies the existence of an isomorphism $S\colon H\to H$ such that 
$\norm{S^{-1}}\norm{S}\leq\norm{\pi}^{2}$ and $S^{-1}w(\cdotp)S\colon C_{\lambda}^{*}(G)\to B(H)$ is a $*$-representation. Equivalently, 
$S^{-1}\pi(\cdotp)S$ is a unitary representation.

(2) We cannot expect an extension of Theorem \ref{4MainThm} for general (= non amenable) groups. 
See  \cite[Chap. 2]{Pis2} for an account on non unitarizable representations of groups on Hilbert space,
and relevant open problems.
\end{remark}

\medskip
\section{Representations of nuclear $C^*$-algebras on spaces with property $(\alpha)$}
We say that a Banach space $X$ has propery $(\alpha)$ if there is a constant $\alpha\geq 1$ such that
\begin{equation}\label{NAlpha(X)}
\Bignorm{\sum_{i,j}\varepsilon_i\otimes\varepsilon_j\otimes t_{ij} 
x_{ij}}_{\rara{X}}\,\leq\,\alpha\,\sup_{i,j}\vert t_{ij}\vert\, 
\Bignorm{\sum_{i,j}\varepsilon_i\otimes\varepsilon_j\otimes x_{ij}}_{\rara{X}}
\end{equation}
for any finite families $(x_{ij})_{i,j}$ in $X$ and $(t_{ij})_{i,j}$ in $\Cdb$. 
This class was introduced in \cite{Pis} and has played an important role in 
several recent issues concerning functional calculi and unconditionality (see \cite{CPSW,DR1,DSW,KaW1,KL}). 
We note that  Banach spaces with property $(\alpha)$ have a finite cotype (because they cannot contain the 
$\ell^{\infty}_n$'s uniformly). Thus Rademacher averages and Gaussian averages 
are equivalent on them. Hence $R$-boundedness and $\gamma$-boundedness (as well as matricial
$R$-boundedness and matricial $\gamma$-boundedness) are equivalent notions on these spaces.
The class of spaces with property $(\alpha)$ is stable under taking subspaces 
and comprises Banach lattices with a finite cotype. 
On the opposite, non trivial noncommutative $L^p$-spaces do not belong to this class.
For a space $X$ with property $(\alpha)$ we let $\alpha(X)$ denote the 
smallest constant $\alpha$ satisfying $(\ref{NAlpha(X)})$. 

Let $A$ be a $C^*$-algebra and let $w\colon A\to B(X)$ be a bounded homomorphism. 
Assume that $X$ has property $(\alpha)$.
It was shown in \cite[Cor. 2.19]{DR2} that if  $A$ is abelian, then $w$ is automatically $R$-bounded.
By \cite{KL}, $w$ is actually matricially $R$-bounded. When $G$ is an amenable group, the $C^*$-algebra
$C^{*}_\lambda(G)$ is nuclear (see e.g. \cite[(1.31)]{Pat}). Thus in view of Theorem \ref{4MainThm}, the question whether
any bounded homomomorphism $w\colon A\to B(X)$ is automatically $R$-bounded
(or matricially $R$-bounded) when $A$ is nuclear became quite relevant. A positive answer to this question
was shown to me by \'Eric Ricard. I thank him for letting me include this result in the present paper.

\begin{theorem}\label{NMain} Let $X$ be a Banach space with property $(\alpha)$ and let $A$ be a nuclear $C^*$-algebra. 
Any bounded homomorphism $w\colon A\to B(X)$ is matricially $R$-bounded. If further $w$ is nondegenerate, then
$$
\matnormR{w}\leq K_X\norm{w}^{2},
$$
where $K_X\geq 1$ is a constant only depending on $\alpha(X)$.
\end{theorem}

We need two lemmas. 
In the sequel we let $(\varepsilon_j)_{j\geq 1},\, (\theta_i)_{i\geq 1}$ and $(\eta_k)_{k\geq 1}$ 
denote Rademacher sequences. For simplicity we will often use the 
same notations $\varepsilon_j,\theta_i,\eta_k$ to  denote values 
of these variables. We start with a double estimate which will lead to the  result 
stated in Theorem \ref{NMain} in the case when $A$ is finite-dimensional. 
When 
\begin{equation}\label{NFD}
A=\mathop{\oplus}\limits_{k=1}^{N}M_{n_k},
\end{equation} 
we let $(E_{ij}^{k})_{1\leq i,j\leq n_k}$ denote the canonical basis of $M_{n_k}$, for any $k=1,\ldots,N$.

\begin{lemma}\label{NEq1}
Let $X$ be a Banach space with property $(\alpha)$, let $n_{1},\ldots, n_{N}$ be positive integers, and let
$$
w\colon \mathop{\oplus}\limits_{k=1}^{N}M_{n_k}\,\longrightarrow B(X)
$$
be any unital homomorphism. Then for any $x\in X$, we have 
$$
C_{X}^{-1}\norm{w}^{-2}\norm{x}\, \leq\, \Bignorm{\sum_{k=1}^{N}\sum_{j=1}^{n_k}\varepsilon_j\otimes 
\eta_k\otimes w(E_{1j}^{k})x}_{\rara{X}}
\,\leq\,C_{X} \norm{w}^{2} \norm{x},
$$
where $C_X\geq 1$ is a constant only depending on $\alpha(X)$.
\end{lemma}

\begin{proof}
Let $\varepsilon_j=\pm 1, \theta_i=\pm 1$ and $\eta_k=\pm 1$ for $j,i,k\geq 1$.
We let
$$
\Delta_r = \sum_{k=1}^{N}\sum_{j=1}^{n_k} n_{k}^{-\frac{1}{2}}\theta_j\, E_{1j}^{k} \qquad\hbox{and}\qquad 
\Delta_c = \sum_{k=1}^{N}\sum_{i=1}^{n_k} n_{k}^{-\frac{1}{2}}\theta_i\, E_{i1}^{k}.
$$
It is plain that
$$
\norm{\Delta_r} = \norm{\Delta_c}=1 \qquad\hbox{and}\qquad \Delta_r\Delta_c=\sum_{k=1}^{N} E_{11}^{k}.
$$
Since $w$ is a homomorphism, we have
\begin{align*}
w(\Delta_c)\Bigl(\sum_{k,j}\varepsilon_j\eta_k \, w(E_{1j}^{k})x\Bigr)\, & = \,
\Bigl(\sum_{k,i} n_{k}^{-\frac{1}{2}}\theta_i\, w(E_{i1}^{k})\Bigr)\,
\Bigl(\sum_{k,j}\varepsilon_j\eta_k \, w(E_{1j}^{k})x\Bigr)\\ &=\,
\sum_{k,j,i} n_{k}^{-\frac{1}{2}}\varepsilon_j\eta_k\theta_i \, w(E_{ij}^{k})x.
\end{align*}
We deduce that
\begin{equation}\label{NEq2}
\Bignorm{\sum_{k,j,i} n_{k}^{-\frac{1}{2}}\varepsilon_j\eta_k\theta_i \, w(E_{ij}^{k})x}\,\leq\,\norm{w} 
\Bignorm{\sum_{k,j}\varepsilon_j\eta_k \, w(E_{1j}^{k})x}.
\end{equation}
Continuing the above calculation, we obtain further that 
\begin{align*}
w(\Delta_r)\Bigl(\sum_{k,j,i} n_{k}^{-\frac{1}{2}}\varepsilon_j\eta_k\theta_i \, w(E_{ij}^{k})x\Bigr)\, & =\, 
w(\Delta_r\Delta_c)\Bigl(\sum_{k,j}\varepsilon_j\eta_k \, w(E_{1j}^{k})x\Bigr)\\ &=\, 
\Bigl(\sum_k w(E_{11}^{k})\Bigr)\,\Bigl(\sum_{k,j}\varepsilon_j\eta_k \, w(E_{1j}^{k})x\Bigr)\\ & =\, 
\sum_{k,j}\varepsilon_j\eta_k \, w(E_{1j}^{k})x.
\end{align*}
Consequently,
\begin{equation}\label{NEq3}
\Bignorm{\sum_{k,j}\varepsilon_j\eta_k \, w(E_{1j}^{k})x}\,\leq\,\norm{w} 
\Bignorm{\sum_{k,j,i} n_{k}^{-\frac{1}{2}}\varepsilon_j\eta_k\theta_i \, w(E_{ij}^{k})x}.
\end{equation}
Now let $U=[u_{ij}^{1}]\oplus\cdots\oplus [u_{ij}^{N}]$ be a fixed unitary of $\mathop{\oplus}\limits_{k=1}^{N}M_{n_k}$.
Then consider the diagonal (unitary) elements 
$$
V=\sum_{k=1}^{N}\sum_{i=1}^{n_k}\eta_k \theta_i\, E_{ii}^{k} \qquad\hbox{and}\qquad 
W= \sum_{k=1}^{N}\sum_{j=1}^{n_k}\varepsilon_j\, E_{jj}^{k}.
$$
Then $VUW$ is a unitary and
$$
w(VUW)x\,=\,\sum_{k,j,i}\varepsilon_j \eta_k \theta_i\, u_{ij}^{k}\, w(E_{ij}^{k})x.
$$
Since $w$ is unital, we deduce that
\begin{equation}\label{NEq4}
\norm{w}^{-1}\norm{x}\leq\,\Bignorm{\sum_{k,j,i}\varepsilon_j \eta_k \theta_i\, u_{ij}^{k} \, w(E_{ij}^{k})x}\, \leq \norm{w}\norm{x}.
\end{equation}
Let us apply the above with the special unitary $U$ defined by 
$$
u_{ij}^{k}\,=\,n_{k}^{-\frac{1}{2}}\, {\rm exp} 
\Bigl\{\frac{2\pi\sqrt{-1}}{n^k}(ij)\Bigr\},
\qquad k=1,\ldots, N,\ i,j=1,\ldots, n_k.
$$
Its main feature is that $\vert u_{ij}^{k}\vert = n_{k}^{-\frac{1}{2}}$
for any $i,j,k$.
Since $X$ has property $(\alpha)$, this implies that for some 
constant $C_X\geq 1$ only depending on $\alpha(X)$, we have
$$
\Bignorm{\sum_{k,j,i} n_{k}^{-\frac{1}{2}} \varepsilon_j \otimes \eta_k\otimes 
\theta_i\otimes w(E_{ij}^{k})x}_{\rarara{X}}\,\leq C_X\, 
\Bignorm{\sum_{k,j,i}\varepsilon_j \otimes\eta_k\otimes \theta_i\otimes u_{ij}^{k} 
w(E_{ij}^{k})x}_{\rarara{X}}
$$
and
$$
\Bignorm{\sum_{k,j,i}\varepsilon_j \otimes\eta_k\otimes \theta_i\otimes u_{ij}^{k} 
w(E_{ij}^{k})x}_{\rarara{X}}
\,\leq C_X\, 
\Bignorm{\sum_{k,j,i}n_{k}^{-\frac{1}{2}} \varepsilon_j 
\otimes\eta_k\otimes \theta_i\otimes w(E_{ij}^{k})x}_{\rarara{X}} .
$$
Combining  with (\ref{NEq2}), (\ref{NEq3}) and (\ref{NEq4}), we get the result.
\end{proof}

For any integer $m\geq 1$, we let
$$
\sigma_{m,X}\colon M_{m}\longrightarrow B\bigl({\rm Rad}_{m}(X)\bigr)
$$
be the canonical homomorphism defined by letting $\sigma_{m,X}(a) =a\otimes I_X$ for any $a\in M_m$.
According to \cite[Lem. 4.3]{KL}, the mappings $\sigma_{m,X}$ are uniformly $R$-bounded. 
The same proof shows they are actually uniformly matricially $R$-bounded. We record this fact for further use.

\begin{lemma}\label{NAlpha}
Let $X$ be a Banach space with property $(\alpha)$. Then
$$
D_X : =\sup_{m\geq 1}\matnormR{\sigma_{m,X}}\,<\,\infty\,.
$$
\end{lemma}

\begin{proof}[Proof of Theorem \ref{NMain}]
Throughout we let $w\colon A\to B(X)$ be a bounded homomorphism. By standard arguments, 
it will suffice to consider the case when $w$ is nondegenerate. The proof will be divided into three steps.

\smallskip\noindent
{\it First step: we assume that $A$ is finite-dimensional,  $w$ is unital and  $\norm{w}=1$}. 
Thus (\ref{NFD}) holds for some positive integers 
$n_1,\ldots, n_N$. Let $m=n_1+\cdots+n_N$, so that $A\subset M_m$ in a canonical way.
Let $(\varepsilon_{jk})_{j,k\geq 1}$ be a doubly indexed family of independent Rademacher variables, and let
$$
S\colon X\longrightarrow {\rm Rad}_{m}(X)
$$
be defined by 
$$
S(x)=\sum_{k=1}^{N}\sum_{j=1}^{n_k} \varepsilon_{jk}\otimes w(E_{1j}^{k})x,\qquad x\in X.
$$
Let $Y\subset {\rm Rad}_{m}(X)$ be the range of $S$. According to Lemma \ref{NEq1} and the 
assumption that $X$ has property $(\alpha)$, $S$ is an isomorphism onto $Y$ and there exist 
a constant $B_X\geq 1$ only depending on $\alpha(X)$ such that 
\begin{equation}\label{NS}
\norm{S}\leq B_X\qquad\hbox{and}\qquad \norm{S^{-1}\colon Y\to X}\leq B_X.
\end{equation}
Let $a=[a_{ij}^{1}]\oplus\cdots\oplus [a_{ij}^{N}]\in A$. For any  $x\in X$, we have
$$
\bigl[\sigma_{m,X} (a)\bigr]\bigl(S(x)\bigr)=\,\sum_{k,j,i}\varepsilon_{ik}\otimes a_{ij}^{k} w(E_{1j}^{k})x.
$$
On the other hand  we have for any $k,i$ that
$E_{1i}^{k} a =\sum_j a_{ij}^{k} E_{1j}^{k}$. 
Hence
$$
w(E_{1i}^{k}) w(a)x = \sum_j a_{ij}^{k} w(E_{1j}^{k})x,
$$
and then
$$
\sum_{k,j,i}\varepsilon_{ik}\otimes a_{ij}^{k} w(E_{1j})x = \sum_{k,i}\varepsilon_{ik}\otimes 
w(E_{1i}^{k}) w(a)x = S\bigl(w(a)x\bigr).
$$
This shows that $\sigma_{m,X}(a) S=S w(a)$. Thus $Y$ is invariant under the action of ${\sigma_{m,X}}_{\vert A}$ and if we let 
$\sigma\colon A\to B(Y)$ be the homomorphism induced by $\sigma_{m_,X}$, we have shown that 
$$
w(a) = S^{-1} \sigma(a) S,\qquad a\in A.
$$
Appealing to (\ref{NS}), this implies that 
$$
\matnormR{w}\leq\norm{S^{-1}}\norm{S}\matnormR{\sigma}\leq \norm{S^{-1}}\norm{S}\matnormR{\sigma_{m,X}}
\leq B_{X}^{2} D_X.
$$

\smallskip\noindent
{\it Second step: we merely assume that $A$ is finite-dimensional and $w$ is unital}.
Let $\U$ be the unitary group of $A$
and let $d\tau$ denote the Haar measure on $\U$.
We define a new norm on $X$ by letting
$$
\triple{x}=\,\Bigl(\int_{\footnotesize{\U}}\norm{w(U)x}^{2}\, d\tau(U)\,\Bigr)^{\frac{1}{2}}, \qquad x\in X.
$$
Since $w$ is unital, this is an equivalent norm on $X$ and 
\begin{equation}\label{NTriple}
\norm{w}^{-1}\norm{x}\,\leq\,\triple{x}\,\leq\,\norm{w}\norm{x},\qquad x\in X.
\end{equation}
Let $\widetilde{X}$ be the Banach space $(X,\triple{\ \cdotp\ })$ and let 
$\widetilde{w}\colon A\to B(\widetilde{X})$ be induced by $w$.
It readily follows from 
(\ref{NTriple}) that 
$$
\matnormR{w}\leq\norm{w}^{2}\matnormR{\widetilde{w}}.
$$
Using Fubini's Theorem it is easy to see that we further have
$$
\alpha(\widetilde{X})\leq \alpha(X).
$$
The first step  shows that we have
$\matnormR{\widetilde{w}} \leq K$  for some constant $K$ only depending on 
$\alpha(\widetilde{X})$. The above observation shows that $K$ does 
actually depend only on $\alpha(X)$, and we therefore obtain an estimate 
$\matnormR{w}\leq K_X\norm{w}^{2}$.

\smallskip\noindent
{\it Third step: $A$ is infinite dimensional and $w$ is nondegenerate}.
We will use second duals in a rather standard way. However the fact that $X$ 
may not be reflexive leads to some technicalities. Observe that using 
Connes's Theorem \cite{Co} and arguing e.g. as in \cite[p. 135]{Pis2} (see also \cite{L}), 
we may assume that there exists a 
directed net $(A_\lambda)_\lambda$ of finite dimensional von Neumann subalgebras of $A^{**}$ such that
$$
A^{**}=\,\overline{\bigcup_\lambda A_\lambda}^{w^{*}}.
$$
Let $u\colon A\to B(X^{**})$ be the homomorphism defined by letting $u(a)=w(a)^{**}$ for any $a\in A$. 
According to \cite[Lem. 2.3]{KL}, there exists a 
(necessarily unique) $w^{*}$-continuous homomorphism $\widehat{u}\colon A^{**}\to B(X^{**})$ 
extending $u$. We claim that
$$
\widehat{u}(1)x=x,\qquad x\in X.
$$
Indeed let $(a_t)_t$ be a contractive approximate identity of $A$ 
and note that since $w$ is nondegenerate, $w(a_t)$ converges strongly to $I_X$.
This implies that $u(a_t)x=w(a_t)x\to x$. Since $a_t\to 1$ in the $w^*$-topology of $A^{**}$, we also
have that $u(a_t)x\to \widehat{u}(1)x$  weakly, which yields the above equality.

Let $Z\subset X^{**}$ be  the range of the projection $\widehat{u}(1)\colon X^{**}\to X^{**}$. 
The above property means that $X \subset Z$.
For any $\lambda$, we let $\widehat{u}_\lambda\colon A_\lambda\to B(Z)$ 
denote the unital homomorphism induced by the restriction of $\widehat{u}$ to $A_\lambda$. 
Since $X$ has property $(\alpha)$, its second dual $X^{**}$ has property $(\alpha)$ as well and
$\alpha(X^{**})=\alpha(X)$, by (\ref{NSecond}). 
Moreover $\norm{\widehat{u}_{\lambda}}\leq \norm{\widehat{u}}=\norm{u}=\norm{w}$. 
Hence by the second step of this proof,
we have a uniform estimate
\begin{equation}\label{NUniform}
\matnormR{\widehat{u}_\lambda}\leq  K_X\norm{w}^{2}.
\end{equation}
Consider $[a_{ij}]\in M_n(A)$ and assume that $\norm{[a_{ij}]}\leq 1$. 
Let us regard $[a_{ij}]$ as an element of $M_n(A^{**})$. Then by Kaplansky's density 
Theorem (see e.g. \cite[Thm. 5.3.5]{KR}), there exist a net $(\lambda_s)_s$ and, for any $s$, a matrix
$[a_{ij}^{s}]$ belonging to the unit ball of $M_n(A_{\lambda_s})$, such that for any 
$i,j=1,\ldots,n$, $a_{ij}^{s}\to a_{ij}$ in the $w^{*}$-topology of $A^{**}$. 
Then for any $x_1,\ldots,x_n$ in $X$ and 
$\varphi_1,\ldots,\varphi_n$ in $X^{*}$, we have
$$
\lim_s \sum_{i,j} \langle \varphi_i, \widehat{u}_{\lambda_{s}}(a_{ij}^{s})x_j\rangle\, =\, 
\sum_{i,j} \langle \varphi_i, w(a_{ij})x_j\rangle\,.
$$
Applying (\ref{NUniform}) we deduce  that
$$
\Bignorm{\sum_{i,j}\varepsilon_i\otimes w(a_{ij}) x_j}_{{\rm Rad}(X)}\,\leq \, K_X\norm{w}^{2}
\Bignorm{\sum_j\varepsilon_j\otimes x_j}_{{\rm Rad}(X)}.
$$
\end{proof}

\begin{remark} 
When $X=H$ is a Hilbert space, the above proof yields $K_H=1$, and we recover the classical 
result that any bounded homorphism $u\colon A\to B(H)$ on a nuclear $C^*$-algebra
is completely bounded, with $\cbnorm{u}\leq \norm{u}^{2}$ (see \cite{B,C, Pis2}). 
\end{remark}

\begin{remark} Let $\norm{\ }_\gamma$ be a cross-norm on $\ell^2\otimes\ell^2$ (in the sense that 
$\norm{z_1\otimes z_2}_\gamma=\norm{z_1}\norm{z_2}$ for all $z_1,z_2$ in $\ell^2$)
and let $\ell^2\otimes_\gamma\ell^2$ denote the completion of the normed space 
$(\ell^2\otimes\ell^2, \norm{\ }_\gamma)$. 
Assume moreover that any bounded operator $a\colon\ell^2\to\ell^2$ 
has a bounded tensor extension
$a\otimes I_{\ell^2}\colon 
\ell^2\otimes_\gamma\ell^2\to\ell^2\otimes_\gamma\ell^2$. 
It follows from the above results that if the Banach space
$\ell^2\otimes_\gamma\ell^2$ has property $(\alpha)$, 
then $\norm{\ }_\gamma$ is equivalent to the
Hilbert tensor norm $\norm{\ }_2$, and hence 
$$
\ell^2\otimes_\gamma\ell^2\,\approx\, S^2,
$$
the space of Hilbert-Schmidt operators on $\ell^2$.
Indeed by the closed graph theorem, there is a constant $K\geq 1$ such that 
$\norm{a\otimes I_{\ell^{2}}}\leq K\norm{a}$ for any 
$a\in B(\ell^2)$. Let
$w\colon B(\ell^2)\to B(\ell^2\otimes_\gamma\ell^2)$ be the 
bounded homomorphism defined by $w(a) =a\otimes I_{\ell^2}$. 
According to Lemma \ref{NEq1}, there is a constant $C\geq 1$ such that for any $n\geq 1$,
$$
C^{-1}\norm{x}\,\leq\, \Bignorm{\sum_{k=1}^n\varepsilon_k
\otimes w(E_{1 k})x}_{\ra{\ell^2\otimes_\gamma\ell^2}}\,\leq\, C\norm{x}
$$
whenever $x$ is a linear combination of the $e_i\otimes e_j$, with $1\leq i,j\leq n$.
For any scalars 
$(s_{ij})_{1\leq i,j\leq n}$ and any $\varepsilon_k=\pm 1$, we have
$$
\sum_{k=1}^n\varepsilon_k w(E_{1k}) \Bigl(\sum_{i,j=1}^{n} s_{ij}\,
 e_i\otimes e_j\Bigr) = e_1\otimes \Bigl(\sum_{i,j=1}^n\varepsilon_i  s_{ij} e_j\Bigr).
$$
Hence for $x=\sum_{i,j=1}^{n} s_{ij}\, e_i\otimes e_j$, we have
\begin{align*}
\Bignorm{\sum_{k=1}^n\varepsilon_k
\otimes w(E_{1 k})x}_{\ra{\ell^2\otimes_\gamma\ell^2}} \, & =\, \Bignorm{
\sum_{i=1}^n\varepsilon_i\otimes \Bigl(\sum_{j=1}^n s_{ij}\, e_j\Bigr)}_{\ra{\ell^2}}\\ & =\, 
\Bigl(\sum_{i=1}^n\Bignorm{\sum_{j=1}^n s_{ij} e_j}^{2}\Bigr)^{\frac{1}{2}}\\ & =\,\Bigl(\sum_{i,j=1}^{n}\vert s_{ij}\vert^{2}\Bigr)^{\frac{1}{2}}\,=\, \norm{x}_{2}.
\end{align*}
This shows that $\norm{x}\approx\norm{x}_2$ and the result follows by density. 

That result is a variant of \cite[Thm 2.2]{KP}, a classical unconditional characterization of $S^2$.
\end{remark}

\medskip
\section{Examples and applications}
In the case when $X$ has property $(\alpha)$, Theorem \ref{NMain} leads to a simplied version of Theorem
\ref{4MainThm}, as follows.

\begin{corollary}\label{5Alpha} Let $G$ be an amenable group and assume that $X$
has property $(\alpha)$. Let $\pi\colon G\to B(X)$ be a bounded
representation. Then $\pi$ is $R$-bounded if and only if it
extends to a bounded homomorphism $w\colon
C^{*}_{\lambda}(G)\to B(X)$.
\end{corollary}

\begin{proof} Since $G$ is amenable, the $C^*$-algebra $C^{*}_{\lambda}(G)$ is nuclear. Hence 
any bounded homomorphism $w\colon C^{*}_{\lambda}(G)\to B(X)$ is $R$-bounded, by 
Theorem \ref{NMain}. The equivalence therefore follows from Theorem \ref{4MainThm}.
\end{proof}

The following is a noncommutative generalization of the fact that if $G$ is an infinite abelian group $G$
and $p\not=2$, there exist bounded functions $\widehat{G}\to \Cdb$ which are not bounded Fourier 
multipliers on $L^p(G)$.

\begin{corollary}\label{5Regular} Let $G$ be an infinite amenable group and let $1\leq p <\infty$.
Let $\lambda_p\colon G\to B(L^{p}(G))$ be the `left regular
representation' defined by letting $\bigl[\lambda_p(t)f](s)
=f(t^{-1}s)$ for any $f\in L^{p}(G)$. Then $\lambda_p$ extends to
a bounded homomorphism $C^{*}_\lambda(G)\to B(L^{p}(G))$ (if and) only if
$p=2$.
\end{corollary}

\begin{proof} Assume that $\lambda_p$ has an extension to
$C^{*}_\lambda(G)$. Since $L^p(G)$ has property $(\alpha)$, Corollary \ref{5Alpha} ensures that 
$\{\lambda_p(t)\, :\ t\in G\}$ is $R$-bounded. According to \cite[Prop. 2.11]{DR2}, this implies that 
$p=2$. (The latter paper considers abelian groups only but the proof works as well in the non abelian case.)
\end{proof}

\bigskip
We will now focus on the three classical groups $\Zdb$, $\Rdb$ and $\Tdb$. 
We wish to mention the remarkable work of Berkson, Gillespie and Muhly 
\cite{BG, BGM} on bounded representations of these groups on UMD Banach spaces. 
Roughly speaking, their results
say that when $G=\Zdb$, $\Rdb$ or $\Tdb$, and $X$ is UMD, any bounded 
representation $\pi\colon G\to B(X)$   
gives rise to a spectral family $E_\pi$ of projections allowing a natural spectral decomposition of $\pi$ 
(see \cite{BG, BGM} for a precise statement).
According to Remark \ref{4Abelian}, our results imply that if $\pi\colon G\to B(X)$ 
is actually $\gamma$-bounded, then $E_\pi$ is induced by a spectral measure.

Representations $\pi\colon \Zdb\to B(X)$ are of the form $\pi(k)=T^{k}$, 
where $T\colon X\to X$ is a bounded invertible operator. Furthermore 
$C^{*}_\lambda(\Zdb)$ coincides with $C(\Tdb)$. In the next statement, 
we let $\kappa\in C(\Tdb)$ be the function defined by $\kappa(z)=z$, and 
we let $\sigma(T)$ denote the spectrum of $T$. We refer to \cite{D} for 
some background on spectral decompositions and scalar type operators.

\begin{proposition}\label{5T} Let $T\colon X\to X$ be a bounded invertible operator.

\begin{itemize}
\item [(1)] The set $\{ T^k\, :\, k\in\Zdb\}$ is $\gamma$-bounded if and only if 
there exists a $\gamma$-bounded unital homomorphim $w\colon C(\Tdb)\to B(X)$ such that $w(\kappa)=T$.
\item [(2)] Assume that $X$ has property $(\alpha)$. Then the following are equivalent.
\begin{itemize}
\item [(i)] The set  $\{ T^k\, :\, k\in\Zdb\}$ is $R$-bounded.
\item [(ii)] There is a  bounded unital homomorphim $w\colon C(\Tdb)\to B(X)$ 
such that $w(\kappa)=T$.
\item [(iii)] $T$ is a scalar type spectral operator and $\sigma(T)\subset\Tdb$.
\end{itemize}
\end{itemize}
\end{proposition}

\begin{proof} Part (1) corresponds to Theorem \ref{4MainThm} when $G=\Zdb$ and in part (2), the 
equivalence between (i) and (ii) is given by Corollary \ref{5Alpha}. The implication
`(iii)$\,\Rightarrow\,$(ii)' follows from \cite[Thm. 6.24]{D}. Conversely, assume (ii). Then 
by \cite[Lem. 3.8]{KL}, $\sigma(T)\subset\Tdb$ and there is a bounded unital homomorphism 
$v\colon C(\sigma(T))\to B(X)$ (obtained by factorizing $w$ through its kernel) such that 
$v(\kappa)=T$, $\sigma(v(f))=f(\sigma(T))$ for any $f\in C(\sigma(T))$, and $v$ is an isomorphism onto its range. 
Since $X$ has property $(\alpha)$, it cannot contain $c_0$. Hence by \cite[VI, Thm. 15]{DU}, any bounded map
$C(\sigma(T))\to X$ is weakly compact. Applying \cite[Thm. 6.24]{D}, we deduce the assertion (iii).
\end{proof}

Turning to representations of the real line, let $(T_t)_{t\in\footnotesize{\Rdb}}$ 
be a bounded $c_0$-group on $X$, and let $A$ denote its infinitesimal generator. It 
spectrum $\sigma(A)$ is included in the
imaginary axis $i\Rdb$. Let ${\rm Rat}\subset C_0(\Rdb)$ denote the subalgebra of all 
rational functions $g$ with poles lying outside the real line and such that ${\rm deg}(g)\leq -1$. 
Rational functional calculus yields a natural definition
of $g(iA)$ for any such $g$. The following is the analog of 
Proposition \ref{5T} for the real line
and has an identical proof. Note that a special case of that result 
is announced in \cite[Cor. 7.6]{W2}, as a consequence of some unpublished work of Kalton and Weis.

\begin{proposition}\label{5R} Let $(T_t)_{t\in\footnotesize{\Rdb}}$ be a bounded $c_0$-group with generator $A$.

\begin{itemize}
\item [(1)] The set $\{ T_t\, :\, t\in\Rdb\}$ is $\gamma$-bounded if and only if there exists a $\gamma$-bounded nondegenerate homomorphim $w\colon C_0(\Rdb)\to B(X)$ such that $w(g)=g(iA)$ for any $g\in {\rm Rat}$.
\item [(2)] Assume that $X$ has property $(\alpha)$. Then the following are equivalent.
\begin{itemize}
\item [(i)] The set  $\{ T_t\, :\, t\in\Rdb\}$ is $R$-bounded.
\item [(ii)] There is a bounded nondegenerate homomorphim $w\colon C_0(\Rdb)\to B(X)$ such that $w(g)=g(iA)$ for any $g\in {\rm Rat}$.
\item [(iii)] $A$ is a scalar type spectral operator.
\end{itemize}
\end{itemize}
\end{proposition}

Let $(X_n)_{n\in\footnotesize{\Zdb}}$ be an unconditional decomposition of a Banach space $X$. 
For any bounded sequence $\theta=(\theta_n)_{n\in\footnotesize{\Zdb}}$ of complex numbers,
let $T_\theta\colon X\to X$ be the associated multiplier operator defined by
$$
T_\theta\Bigl(\sum_n x_n\Bigr)\,=\,\sum_n\theta_n x_n,\qquad x_n\in X_n.
$$
We say that the decomposition $(X_n)_{n\in\footnotesize{\Zdb}}$ is $\gamma$-unconditional
(resp. $R$-unconditional) if the set
$$
\bigl\{T_\theta\,:\,\theta\in\ell^{\infty}_{\footnotesize{\Zdb}},\ \norm{\theta}_\infty\leq 1\bigr\}\, \subset\, B(X)
$$
is $\gamma$-bounded (resp. $R$-bounded).

For any bounded representation $\pi\colon\Tdb\to B(X)$, and any $n\in\Zdb$, we let $\widehat{\pi}(n)$ 
denote the $n$th Fourier coefficient of $\pi$, defined by
$$
\widehat{\pi}(n)\,=\,\frac{1}{2\pi}\,\int_{0}^{2\pi} \pi(t)e^{-int}\, dt\,.
$$
Equivalently, $\widehat{\pi}(n)=\sigma_{\pi}(t\mapsto e^{-int})$. Each $\widehat{\pi}(n)\colon X\to X$ 
is a bounded projection, the ranges $\widehat{\pi}(n)X$ form a direct sum and 
$\oplus_n  \widehat{\pi}(n)X$ is dense in $X$. However $\bigl(\widehat{\pi}(n)X\bigr)_{n\in\footnotesize{\Zdb}}$ is not a Schauder decomposition in general. (Indeed, take $X=L^{1}(\Tdb)$ and let $\pi$ be the regular representation of $\Tdb$ on
$L^{1}(\Tdb)$. Then $\widehat{\pi}(n)f=\widehat{f}(-n) e^{-in\bullet}$ for any $f$, and the Fourier decomposition on $L^{1}(\Tdb)$ is not a Schauder decomposition.)

\begin{proposition}\label{5R} 
Let $\pi\colon\Tdb\to B(X)$ be a bounded representation.

\begin{itemize}
\item [(1)] $\pi$ is $\gamma$-bounded if and only if 
$\bigl(\widehat{\pi}(n)X\bigr)_{n\in\footnotesize{\Zdb}}$ is a $\gamma$-unconditional decomposition of $X$.
\item [(2)] Assume that $X$ has property $(\alpha)$. Then $\pi$ is $R$-bounded if and only if 
$\bigl(\widehat{\pi}(n)X\bigr)_{n\in\footnotesize{\Zdb}}$ is an unconditional decomposition of $X$.
\end{itemize}
\end{proposition}

\begin{proof} Assume that $\pi$ extends to a bounded homomorphism
$w\colon c_{0,\footnotesize{\Zdb}}\to B(X)$. 
Then for any finitely supported scalar sequence $(\theta_n)_{n\in\footnotesize{\Zdb}}$, we have
$$
w\bigl((\theta_n)_{n}\bigr)\,=\,\sum_{n} \theta_{-n}\widehat{\pi}(n).
$$
Since $w$ is nondegenerate and bounded, this implies that 
$\bigl(\widehat{\pi}(n)X\bigr)_{n\in\footnotesize{\Zdb}}$ is an unconditional decomposition of 
$X$. It is clear that $\bigl(\widehat{\pi}(n)X\bigr)_{n\in\footnotesize{\Zdb}}$ is actually 
$\gamma$-unconditional if and only if $w$ is $\gamma$-bounded. 
The result therefore follows from
Theorem \ref{4MainThm} and Corollary \ref{5Alpha}.
\end{proof}

\bigskip
In the last part of this section, we are going to discuss the failure of the equivalence (i)$\,\Leftrightarrow\,$(ii) in Proposition 
\ref{5T}, (2), when $X$ is not supposed to have property $(\alpha)$. We use ideas from 
\cite{DSW} and \cite{KL}. Let $(P_n)_{n\geq 1}$ be a sequence of 
bounded projections on some Banach space $X$. We say that this sequence is unconditional if  $\bigl(P_nX\bigr)_{n\geq 1}$ is an unconditional decomposition of $X$, and we say that $(P_n)_{n\geq 1}$ has property $(\alpha)$ if further there is a constant $\alpha\geq 1$ such that
\begin{equation}\label{5Alpha2}
\Bignorm{\sum_{i,j}\varepsilon_i \otimes t_{ij} P_j(x_i)}_{\ra{X}}\,\leq\,\alpha\,\sup_{i,j}\vert t_{ij}\vert\, 
\Bignorm{\sum_{i,j}\varepsilon_i \otimes P_j(x_i)}_{\ra{X}}
\end{equation}
for any finite families $(x_{j})_{j}$ in $X$ and $(t_{ij})_{i,j}$ in $\Cdb$. 
If $(P_n)_{n\geq 1}$ is unconditional, then we have a uniform equivalence
$$
\Bignorm{\sum_{i,j}\varepsilon_i \otimes P_j(x_i)}_{\ra{X}}\,\approx\,
\Bignorm{\sum_{i,j}\varepsilon_i\otimes\varepsilon_j \otimes P_j(x_i)}_{\rara{X}}.
$$
Hence if $X$ has property $(\alpha)$, any unconditional sequence $(P_n)_{n\geq 1}$ on $X$
has property $(\alpha)$. Conversely, let $P_n\colon\ra{X}\to\ra{X}$ be 
the canonical projection defined by letting
$$
P_n\Bigl(\sum_{j\geq 1} \varepsilon_j\otimes x_j\Bigr)\, =\varepsilon_n\otimes x_n.
$$
Then $(P_n)_{n\geq 1}$ is unconditional on $\ra{X}$ for any $X$, and this sequence has 
property $(\alpha)$ on $\ra{X}$ if and only if $X$ has property $(\alpha)$.

Here is another typical example. For any $1\leq p<\infty$, let $S^p$ denote the Schatten
$p$-class on $\ell^2$ and regard any element of $S^p$ as a bi-infinite matrix
$a=[a_{ij}]_{i,j\geq 1}$ in the usual way. We let $E_{ij}$ denote the matrix units of $B(\ell^{2})$ 
and write $a=\sum_{i,j} a_{ij} E_{ij}\,$ for simplicity. For any $n\geq 1$, let $P_n\colon S^p\to S^p$
be the `$n$th column projection' defined by 
$$
P_n\Bigl(\sum_{i,j}a_{ij} E_{ij}\Bigr)\,=\sum_i a_{in} E_{in}.
$$
It is clear that the sequence $(P_n)_{n\geq 1}$ is unconditional on $S^p$. However if $p\not=2$,
$(P_n)_{n\geq 1}$ does not have property $(\alpha)$. This follows from the lack of unconditionality of 
the matrix decomposition on $S^p$. Indeed, let $a=\sum_{i,j} a_{ij} E_{ij}\,$, let $(t_{ij})_{i,j}$ 
be a finite family of complex numbers and set $x_i=\sum_j a_{ij} E_{ij}\,$ for any $i\geq 1$.
Then 
$$
\Bignorm{\sum_{i,j}\varepsilon_i \otimes  P_j(x_i)}_{\ra{X}}\,=\,\bignorm{[a_{ij}]}_{S^p}\qquad \hbox{and}\qquad
\Bignorm{\sum_{i,j}\varepsilon_i \otimes t_{ij} P_j(x_i)}_{\ra{X}}\,=\,\bignorm{[t_{ij}a_{ij}]}_{S^p}.
$$
Hence (\ref{5Alpha2}) cannot hold true.

\begin{proposition}\label{5CounterExample} Assume that $X$ has a finite cotype
and admits a sequence $(P_n)_{n\geq 1}$ of projections which is unconditional but does not have
property $(\alpha)$. Then there exists an invertible operator $T\colon X\to X$ such that 
the set $\{T^k\, :\, k\in\Zdb\}$ is not $R$-bounded, but there exists a 
bounded unital homomorphism $w\colon C(\Tdb) \to B(X)$ such that $w(\kappa)=T$.
\end{proposition}

\begin{proof} Let $(\zeta_j)_{j\geq 1}$ be a sequence of distinct points of $\Tdb$. 
Since $(P_n)_{n\geq 1}$ is unconditional, one defines a bounded 
unital homomorphism $w\colon C(\Tdb) \to B(X)$ by letting
$$
w(f)\,=\,\sum_{j=1}^{\infty} f(\zeta_j)\, P_j,\qquad f\in C(\Tdb).
$$
Arguing as in \cite[Remark 4.6]{KL}, we obtain that $w$ is not $R$-bounded. 

Let $T=w(\kappa)$, this is an invertible operator. If $\{T^k\, :\, k\in\Zdb\}$ were 
$R$-bounded, then $w$ would be $R$-bounded as well, by Theorem \ref{4MainThm} and the cotype assumption.
\end{proof}

According to the above discussion, Proposition \ref{5CounterExample}
applies on $S^p$ for any $1\leq p\not=2<\infty$, as well as on any space of the form $\ra{X}$ 
when $X$ does not have property $(\alpha)$ but has a finite cotype. 
This leads to the following general question: 

{\it When $X$ does not have property $(\alpha)$,
find a characterization of bounded invertible operators 
$T\colon X\to X$ such that $\pi\colon k\in\Zdb\mapsto T^k$ extends to a bounded homomorphism $C(\Tdb)\to B(X)$.}

\vskip 1cm
\noindent
{\bf Acknowledgements.} I wish to thank \'Eric Ricard for showing me a proof of Theorem \ref{NMain} and for
several stimulating discussions. I also thank the referee for his comments and 
the careful reading of the manuscript.

\end{document}